\documentclass[draft,10pt]{article}

\usepackage{amsmath}
\usepackage{amsthm}
\usepackage{amssymb}
\usepackage{latexsym}
\usepackage{color}

\newtheorem{thm}{Theorem}[section]
\newtheorem{lem}[thm]{Lemma}

\theoremstyle{definition}

\newcommand{\F}{\mathbb{F}}

\newcommand{\Aut}{\mathrm{\mathop{Aut}}}
\newcommand{\GL}{\mathrm{\mathop{GL}}}
\renewcommand{\S}{\mathcal S}
\renewcommand{\char}{\text{char}\, }
\newcommand{\relstar}{\,\stackrel{*}{\sim}\,}
\newcommand{\relstarplus}{\,\stackrel{*+}{\sim}\,}
\newcommand{\relpsi}{\,\stackrel{\psi}{\sim}\,}
\renewcommand{\leq}{\leqslant}

\begin{document}

\title{Six-dimensional nilpotent Lie algebras}
\author{Serena Cical\`o, Willem A. de Graaf and Csaba Schneider}
\date{1 June 2020}
\maketitle

\begin{abstract}
We give a full classification of 6-dimensional nilpotent Lie algebras over an arbitrary field, including fields that are not algebraically closed and fields of characteristic~2. To achieve the classification we use the action of the automorphism group on the second cohomology space, as isomorphism types of nilpotent Lie algebras correspond to orbits of subspaces under this action. In some cases, these orbits are determined using geometric invariants, such as the Gram determinant or the Arf invariant. As a byproduct, we completely determine, for a 4-dimensional vector space $V$, the orbits of $\GL(V)$ on the set of 2-dimensional subspaces of $V\wedge V$. 
\end{abstract}

\bigskip


\bigskip


\bigskip

\section{Introduction}

The classification of small-dimensional Lie algebras is a classical problem. The history of the classification problem
of 6-dimensional nilpotent Lie algebras goes back to Umlauf (\cite{umlauf}). In the 1950's Morozov (\cite{morozov}) published a classification of 6-dimensional nilpotent Lie algebras valid over fields of characteristic 0. Recently several classifications have appeared, over various ground fields. We mention \cite{gong} (over algebraically closed fields, and over the real field), \cite{sch} (over various finite fields) \cite{artw} (over fields of characteristic not 2). However, no classification that treats all ground fields, in particular fields of characteristic 2, is known up to now. It is the purpose of this paper to complete the classification
of 6-dimensional nilpotent Lie algebras over an arbitrary field. \par 
Nilpotent Lie algebras up to dimension five are well-known. There is just one isomorphism type of nilpotent Lie algebras with dimension two, two  isomorphism types in dimension~3, three isomorphism type in dimension 4, and 9 isomorphism types in dimension~5. The classification of nilpotent Lie algebras with dimension up to~5 is independent of the field in the sense that the isomorphism types can be described by uniform structure constant tables with integer entries. This is not, however, the case in dimension~6, as the number of isomorphism types of 6-dimensional nilpotent Lie algebras may depend on the characteristic of the underlying field. There are 36 isomorphism types over $\F_2$, but only 34 isomorphism types over $\F_3$; see~\cite{sch}. Over fields $\F$ of characteristic not 2, the number of isomorphism types is $26+4s$ where $s$ is the (possibly infinite) 
index of $(\F^*)^2$ in the multiplicative group
$\F^*$ of $\F$ (\cite{artw}). In the present paper we give a classification that covers all ground fields, in particular also those of characteristic 2. For a field $\F$ of characteristic~2, define an equivalence relation $\relstarplus$ on $\F$ as follows:
$\alpha\relstarplus\beta$ if and only if $\alpha=\gamma^2\beta+\delta^2$ with some $\gamma\in\F^*$ and  $\delta\in\F$. Let $t$ be the (possibly infinite) number of 
 equivalence classes of $\relstarplus$ in $\F$. Moreover, for a field $\F$ of characteristic~2, define the
 equivalence relation $\relpsi$ by $\alpha\relpsi\beta$ if and only if there is a $\gamma\in \F$ with $\gamma^2+\gamma+\alpha+\beta=0$. Let $u$ be the (possibly infinite) number of equivalence classes of
 $\relpsi$ in $\F$.

\begin{thm}\label{main}
The number of isomorphism types of nilpotent Lie algebras with dimension~$6$ over fields of characteristic different from~$2$ is $26+4s$, while this number is $26+2s+4t+2u$ over fields of characteristic~$2$.
\end{thm}

In Section~\ref{ressect}, a list of the isomorphism classes of the 6-dimensional nilpotent Lie algebras over an arbitrary field is given. The first part of this theorem concerning fields of characteristic different from~2 was already proved in the article~\cite{artw} by the second author (see also~\cite{prew} for more details).

As stressed above, our paper treats fields of characteristic~2, and, for the first time, gives a full classification of 6-dimensional nilpotent Lie algebras over an arbitrary field of characteristic~2. As far as we are aware, in characteristic~2, the only existing classification of such Lie algebras was given over algebraically closed fields by Gong's Ph.D.\ dissertation~\cite{gong}. Comparing Gong's results to ours, we found that Gong's classification contains one mistake: namely, his Lie algebras $N_{6,2,10}$ and $(E)$ are both isomorphic to our Lie algebra $L_{6,24}(0)$ defined in Section~\ref{ressect}. Apart from this, our classification agrees with Gong's.\par

The original aim of the research presented in this paper was to extend the results of~\cite{artw} to fields of characteristic~2. As in the course of this work some proofs of~\cite{artw} were revised, we decided, in this paper, to present a full classification of 6-dimensional nilpotent Lie algebras that is valid over all fields. In addition, some results in~\cite{artw} relied on computer calculations (specifically, computing a Gr\"obner bases for  ideals in a polynomial rings), while the arguments of the present paper are all theoretical with no computer calculations involved. Nevertheless, we should mention here that several of the theoretical arguments in Section~\ref{orbsect} were inspired by Gr\"obner basis computations in the computational algebra system Magma~\cite{Magma} and it would have been significantly more difficult, maybe even impossible, to obtain the classification in Theorem~\ref{main} without performing such computations.\par
Our methodology, which is the same as in~\cite{artw}, is explained in Section~\ref{backg}. We construct the 6-dimensional nilpotent Lie algebras as certain central extensions, descendants in our terminology, of lower-dimensional algebras. To separate the isomorphism classes of the descendants, we use the action of the automorphism group on the subspaces of the second cohomology space. It is interesting to note that, in some examples, the automorphism group preserves a non-degenerate quadratic form on the second cohomology space and the isomorphism classes of the descendants can be characterized by purely geometric means. An example of this situation is the abelian Lie algebra $L$ with dimension~4, whose automorphism group $\GL(4,\F)$, by the Klein correspondence, preserves a quadratic form on the second cohomology space $H^2(L,\F)=(\F^4)\wedge(\F^4)$ with dimension~6. In characteristic different from~2, the 6-dimensional descendants of $L$ can be completely determined using the Gram determinant of the restriction of the quadratic form to the 2-dimensional subspaces of $H^2(L,\F)$. In characteristic~2, we determined these descendants using the Arf invariant of the restriction of the quadratic form to these 2-dimensional subspaces. See Sections~\ref{formsect} and~\ref{orbsect} for the details.
The use of the Klein correspondence in the classification of nilpotent Lie algebras of nilpotency class~2 was also explored in~\cite{stroppel}. 
An interesting byproduct of our work is the determination of the $\GL(4,\F)$-orbits on the 2-dimensional subspaces of $(\F^4)\wedge (\F^4)$ 
(Theorem~\ref{kleinth}).

Here is an outline of the paper. In Section \ref{backg} we describe the 
cohomological method that we use to classify nilpotent Lie algebras, which 
also appeared in \cite{artw}, \cite{skjelsund}, \cite{gong}. In Section \ref{ressect} we present the main result of this paper, that is the classification of the 6-dimensional nilpotent Lie algebras. Section \ref{formsect} has a number of results on bilinear and quadratic forms that we need. Then in Section \ref{orbsect} the main work is performed to prove the main result.\par

The main result of this paper can be accessed electronically using the 
{\sf LieAlgDB} package~\cite{liealgdb} of the computational 
algebra system {\sf GAP}~\cite{gap}.

\section{A summary of the method}\label{backg}

The main idea that we use here is to obtain the nilpotent Lie algebras of dimension $n$ as central extensions of Lie algebras of smaller dimension. The central extensions are defined using the second cohomology space, and the isomorphism classes of the central extensions correspond to the orbits of the automorphism group on the set of the subspaces of this cohomology space. This method has been described for Lie algebras by Skjelbred and Sund (\cite{skjelsund}). Similar ideas appear in the recent work concerning the classification of $p$-groups; see, e.g., \cite{eickob}, \cite{newobrvau}, \cite{obrien}, \cite{obvau}. We summarize the method in this section without giving proofs or explanations; the details can, for instance, be found in~\cite{artw}.\par
For a Lie algebra $L$, let $L^i$ denote the terms of the lower central series. If $L$ is nilpotent then $L^{i+1}=0$ with some $i$ and the smallest such $i$ is called the {\em nilpotency class} of $L$. The second term $L^2$ of the lower central series will usually be written as $L'$. We denote the center of $L$ by $C(L)$. Adapting terminology from~\cite{obrien} to our context, a Lie algebra $K$ is said to be a {\em descendant} of the Lie algebra $L$ if $K/C(K)\cong L$ and $C(K)\leq K'$. If $\dim C(K)=s$ then $K$ is also referred to as a {\em step-$s$ descendant}. A descendant of a nilpotent Lie algebra is nilpotent. Conversely, if $K$ is a finite-dimensional nilpotent Lie algebra over a field $\F$, then $K$ is either a descendant of a smaller-dimensional nilpotent Lie algebra, or $K=K_1\oplus\F$ where $K_1$ is an ideal of $K$ and $\F$ is viewed as a 1-dimensional Lie algebra. Hence determining the isomorphism types of the descendants of the nilpotent Lie algebras with dimension at most~5 suffices for the classification of the 
nilpotent Lie algebras with dimension~6.\par 
The main idea of the method is that, for a nilpotent Lie algebra $L$ over a field $\F$, the isomorphism types of the descendants
of $L$ are in 1-1 correspondence with the $\Aut(L)$-orbits of some of the subspaces of the second cohomology space $H^2(L,\F)$. The second cohomology spaces for nilpotent Lie algebras are defined as follows. For a vector space $V$, let $Z^2(L,V)$ denote the set of alternating bilinear maps $\vartheta:L\times L\rightarrow V$ with the property that \[\vartheta([x_1,x_2],x_3)+\vartheta([x_3,x_1],x_2)+\vartheta([x_2,x_3],x_1)=0 \text{ for all $x_1,\ x_2,\ x_3\in L$}.\]
The set $Z^2(L,V)$ is viewed as a vector space over $\F$ and the elements of $Z^2(L,V)$ are said to be {\em cocycles}. We define, for a linear map $\nu \colon L\to V$, a map $\eta_\nu\colon L\times L\rightarrow V$ as $\eta_\nu(x,y)=\nu([x,y])$. The set $\{\eta_\nu\ |\ \mbox{$\nu:L\rightarrow V$ is linear}\}$ is denoted by $B^2(L,V)$. It is routine to check that $B^2(L,V)$ is a subspace of $Z^2(L,V)$, and the elements of $B^2(L,V)$ are called {\em coboundaries}. The {\em second cohomology space} $H^2(L,V)$ is defined as the quotient $Z^2(L,V)/B^2(L,V)$. \par
The vector spaces defined in the previous paragraph can be viewed as $\Aut(L)$ modules. Indeed, for $\varphi\in \Aut(L)$ and $\vartheta\in Z^2(L,V)$ define $\varphi\vartheta\in Z^2(L,V)$ by the equation $(\varphi\vartheta)(x,y) = \vartheta(\varphi(x),\varphi(y))$. The action $\vartheta\mapsto \varphi\vartheta$ makes $Z^2(L,V)$ an $\Aut(L)$-module and it is easy to see that $B^2(L,V)$ is an $\Aut(L)$-submodule. Hence the quotient $H^2(L,V)$ can also be viewed as an $\Aut(L)$-module.\par
Let $L$ be a Lie algebra and $V$ a vector space over a field $\F$. For $\vartheta\in Z^2(L,V)$, define a Lie algebra $L_\vartheta$ 
as follows. The underlying space of $L_\vartheta$ is $L\oplus V$. The product of two elements $x+v,\ y+u\in L_\vartheta$, is defined
as $[x+v,y+w]=[x,y]_L+\vartheta(x,y)$ where $[x,y]_L$ denotes the product in $L$. Then $L_\vartheta$ is a Lie algebra and $V$ is an ideal of $L_\vartheta$ such that $V\leq C(L_\vartheta)$.  In addition, $L\cong L_\vartheta/V$, and hence $L_\vartheta$ is a central extension of $L$. Further, if $\vartheta_1,\ \vartheta_2\in Z^2(L,V)$ such that $\vartheta_1-\vartheta_2\in B^2(L,V)$ then $L_{\vartheta_1}\cong L_{\vartheta_2}$, and so the isomorphism type of $L_\vartheta$ only depends on the element $\vartheta+B^2(L,V)$ of $H^2(L,V)$. Conversely let $K$ be a Lie algebra such that $C(K)\neq 0$, and set $V=C(K)$ and $L=K/C(K)$. Let $\pi \colon K\to L$ be the projection map. Choose an injective linear map $\sigma \colon L\to K$ such that $\pi(\sigma(x))=x$ for all $x\in L$. Define $\vartheta \colon L\times L\to V$ by $\vartheta(x,y) = [ \sigma(x), \sigma(y) ] -\sigma([x,y])$. Then $\vartheta$ is a cocycle such that $K\cong L_\vartheta$.  Though $\vartheta$ depends on the choice of $\sigma$, the coset $\vartheta+B^2(L,B)$ is independent of $\sigma$. Hence the central extension $K$ of $L$ determines a well-defined element of $H^2(L,V)$. \par
Let us now fix a basis $\{e_1,\ldots,e_s\}$ of $V$. A cocycle $\vartheta \in Z^2(L,V)$ can be written as 
\[\vartheta(x,y) = \sum_{i=1}^s \vartheta_i(x,y) e_i,\]
where $\vartheta_i \in Z^2(L,\F)$. Furthermore, $\vartheta$ is a coboundary if and only if all $\vartheta_i$ are. For $\vartheta\in Z^2(L,V)$, let $\vartheta^\perp$ denote the {\em radical} of $\vartheta$; that is, the set of elements $x\in L$ such that $\vartheta(x,y)=0$ for all $y\in L$. Then 
\[\vartheta^\perp=\bigcap_{\eta\in\langle \vartheta_1,\ldots,\vartheta_s\rangle}
\eta^\perp=\vartheta_1^\perp\cap\cdots\cap\vartheta_s^\perp.\]

\begin{thm}[Lemmas~2--4 in \cite{artw}]\label{coctheorem}
Let $L$ be a Lie algebra, let $V$ be a vector space with fixed basis $\{e_1,\ldots,e_s\}$ over a field $\F$, and let $\vartheta,\ \eta$ be elements of $Z^2(L,V)$. 
\begin{enumerate}
\item[(i)] The Lie algebra $L_\vartheta$ is a step-$s$ descendant of $L$ if and only if $\vartheta^\perp\cap C(L)=0$ and the image of the subspace $\langle \vartheta_1,\ldots,\vartheta_s\rangle$ in $H^2(L,\F)$ is $s$-dimensional .
\item[(ii)]
Suppose that $\eta$ is an other element of $Z^2(L,V)$ and that $L_\vartheta,\ L_\eta$ are descendants of $L$. Then $L_\vartheta\cong L_\eta$ if and only if images of the subspaces $\langle \vartheta_1,\ldots,\vartheta_s\rangle$ and $\langle \eta_1,\ldots,\eta_s\rangle$ in $H^2(L,\F)$ are in the same orbit under the action of $\Aut(L)$. 
\end{enumerate}
\end{thm}

A subspace $U$ of $H^2(L,\F)$ is said to be allowable if $\bigcap_{\vartheta\in U}\vartheta^\perp\cap C(L)=0$. By Theorem~\ref{coctheorem}, there is a one-to-one correspondence between the set of isomorphism types of step-$s$  descendants of $L$ and the $\Aut(L)$-orbits on the $s$-dimensional allowable subspaces of $H^2(L,\F)$. Hence the classification of 6-dimensional nilpotent Lie algebras requires that we determine these orbits for all nilpotent Lie algebras of dimension at most~5. The determination of these orbits is achieved in Section~\ref{orbsect}.

\section{The 6-dimensional nilpotent Lie algebras}\label{ressect}

Let $\F$ be an arbitrary field. In order to classify the 6-dimensional nilpotent Lie algebras over $\F$, we determine the isomorphism classes of the 6-dimensional descendants of the nilpotent Lie algebras with dimension at most 5. In this section
we summarize the result by listing these isomorphism classes for each of the Lie algebras with dimension at most 5, while in the 
Section~\ref{orbsect} we provide with a detailed proof. The notation we use to describe nilpotent Lie algebras of dimension at most~5 is the same as in~\cite{artw}. \par
Unlike the 5-dimensional algebras, nilpotent Lie algebras of dimension~6 cannot be described uniformly over all fields. In some cases, the isomorphism types of descendants will depend on a parameter. In order to describe these cases, we need some notation. First, for a field $\F$, let $\F^*$ denote the multiplicative group of non-zero elements of $\F$. If $\F$ is a field, then let $\relstar$ denote the equivalence relation on $\F$ defined as $\alpha\relstar \beta$ if and only if $\alpha= \gamma^2\beta$ with some $\gamma\in\F^*$. If $\char\F=2$ then define the equivalence relation $\relstarplus$ as $\alpha\relstarplus\beta$ if and only if $\alpha=\gamma^2\beta+\delta^2$ with some $\gamma\in\F^*$ and  $\delta\in\F$. Using that $\char\F=2$, it is easy to show that $\relstarplus$ is indeed an equivalence relation. For a set $X$ and equivalence relation $\sim$, let $X/(\sim)$ denote a transversal
of the equivalence relation $\sim$; that is,  $X/(\sim)$ is a set that contains precisely one element from each of the equivalence classes of $\sim$. Now let $\F$ be a field of characteristic 2. View $\F$ as a vector space over $\F_2$, the map $\psi \colon \F\to\F$ defined by $\psi(x)=x^2+x$ is $\F_2$-linear with kernel $\F_2$. Let $\relpsi$ denote the equivalence relation that corresponds to the 
coset partition of $\F$ with respect to the subspace $\psi(\F)$.
Note that if $\F$ is a finite field, then $\psi(\F)$ is a subspace of codimension~1, and in 
particular $|\F/(\relpsi)|=2$.  On the other hand, if $\F$ is algebraically closed 
then $\psi$ is surjective and $|\F/(\relpsi)|=1$.

Following and extending the notation in~\cite{artw}, 
the nilpotent Lie algebras in this paper are denoted by $L_{d,k}$, 
$L_{d,k}(\varepsilon)$, $L^{(2)}_{d,k}$, or 
$L^{(2)}_{d,k}(\varepsilon)$ 
where $d$ is the dimension of the algebra, $k$ is its index among the nilpotent Lie algebras with dimension $d$, $\varepsilon$ is a possible parameter, 
and the superscript ``$(2)$'' refers to the fact that the algebra
is defined over a field of characteristic~2. 
The list of nilpotent Lie algebras with dimension at most~5, described in the same notation, can be found in~\cite{artw}. In particular, there are 9 isomorphism types of nilpotent Lie algebras of dimension~5: $L_{5,1},\ldots,L_{5,9}$. Hence there are 9 isomorphism types of nilpotent Lie algebras with dimension~6 that are not descendants of smaller-dimensional Lie algebras, namely $L_{6,1},\ldots,L_{6,9}$, where $L_{6,i}=L_{5,i}\oplus\F$.\par

Next we describe the 6-dimensional descendants of nilpotent Lie algebras with dimension at most~5. The Lie algebras in this section are given with  multiplication tables with respect to  fixed bases with trivial products  of the form $[x_i,x_j]=0$ omitted.
With respect to the list of~\cite{artw} we have made a few small changes.
The multiplication table of $L_{6,19}(\varepsilon)$, for $\varepsilon\neq 0$, 
is different from (but isomorphic to) the Lie algebra denoted with the same 
symbol in~\cite{artw}. The Lie algebras $L_{6,19}(0)$  and 
$L_{6,21}(0)$ 
from~\cite{artw} are 
denoted here by $L_{6,27}$ and $L_{6,28}$, respectively. We have made these changes because the structure
of $L_{6,k}(\epsilon)$, $k=19,21$, is different for $\varepsilon=0$ and
$\varepsilon\neq 0$. 
Furthermore, in characteristic~2 a few new algebras appear that are
not contained in~\cite{artw}. These Lie algebras are $L_{6,k}^{(2)}$ or 
$L_{6,k}^{(2)}(\varepsilon)$ (here $k\in\{1,\ldots, 8\}$).

\subsection*{Step-1 descendants of $5$-dimensional Lie algebras}

\begin{itemize}
\item[(5/1)] The abelian Lie algebra $L_{5,1}$ has no step-1 descendants.
\item[(5/2)]
The Lie algebra $L_{5,2}=\langle x_1,\ldots,x_5\mid [ x_1,x_2]=x_3\rangle$ has only one isomorphism class of step-$1$ descendants namely
\[L_{6,10}=\langle x_1,\ldots,x_6\mid [x_1,x_2]=x_3, [x_1,x_3]=x_6, [x_4,x_5]=x_6\rangle.\]
\item[(5/3)]
The Lie algebra $L_{5,3}=\langle x_1,\ldots,x_5\mid [x_1,x_2]=x_3,[x_1,x_3]=x_4\rangle$  has two isomorphism classes of step-$1$ descendants namely
\begin{eqnarray*}
L_{6,11}&=&\langle x_1,\ldots,x_6\mid [x_1,x_2]=x_3, [x_1,x_3]=x_4, [x_1,x_4]=x_6, [x_2,x_3]=x_6,[x_2,x_5]=x_6\rangle\mbox{ and}\\
L_{6,12}&=&\langle x_1,\ldots,x_6\mid [x_1,x_2]=x_3, [x_1,x_3]=x_4, [x_1,x_4]=x_6, 
[x_2,x_5]=x_6\rangle.
\end{eqnarray*}
\item[(5/4)]
The Lie algebra $L_{5,4}=\langle x_1,\ldots,x_5\mid [x_1,x_2]=x_5,[x_3,x_4]=x_5\rangle$ has no step-1 descendants.
\item[(5/5)] 
Let $L_{5,5}=\langle x_1,\ldots,x_5\mid  [x_1,x_2]=x_3,[x_1,x_3]=x_5,[x_2,x_4]=x_5\rangle$. If $\char \F\neq 2$, then $L_{5,5}$ has a unique isomorphism type of step-1 descendants, namely 
\[L_{6,13}=\langle x_1,\ldots,x_6\mid [x_1,x_2]=x_3,[x_1,x_3]=x_5,[x_1,x_5]=x_6,[x_2,x_4]=x_5, [x_3,x_4]=x_6\rangle.\]
If $\char \F=2$ then $L_{5,5}$ has two isomorphism classes of step-1 descendants, namely $L_{6,13}$ above and 
\[L_{6,1}^{(2)}=\langle x_1,\ldots,x_6\mid [x_1,x_2]=x_3,[x_1,x_3]=x_5,[x_1,x_5]=x_6,[x_2,x_4]=x_5+x_6, [x_3,x_4]=x_6\rangle.\]
\item[(5/6)]
Set $L_{5,6}=\langle x_1,\ldots,x_6\mid [x_1,x_2]=x_3,[x_1,x_3]=x_4,[x_1,x_4]=x_5,[x_2,x_3]=x_5\rangle$. If $\char \F\neq 2$, then $L_{5,6}$ has two isomorphism classes of step-1 descendants, namely 
\begin{eqnarray*}
L_{6,14}&=&\langle x_1,\ldots,x_6\mid [x_1,x_2]=x_3, [x_1,x_3]=x_4, [x_1,x_4]=x_5,\\ 
&&{[x_2,x_3]}=x_5, {[x_2,x_5]}=x_6, [x_3,x_4]=-x_6\rangle;\\
L_{6,15}&=&\langle x_1,\ldots,x_6\mid [x_1,x_2]=x_3, [x_1,x_3]=x_4, {[x_1,x_4]}=x_5,\\
&&{[x_1,x_5]}=x_6, [x_2,x_3]=x_5, [x_2,x_4]=x_6\rangle.
\end{eqnarray*}
If $\char \F=2$ then $L_{6,15}$ and the Lie algebras 
\begin{eqnarray*}
L_{6,2}^{(2)}&=&\langle x_1,\ldots,x_6\mid [x_1,x_2]=x_3,[x_1,x_3]=x_4,[x_1,x_4]=x_5,\\
&&{[x_1,x_5]}=x_6,{[x_2,x_3]}=x_5+ x_6,[x_2,x_4]=x_6\rangle;\\
L_{6,3}^{(2)}(\varepsilon)&=&\langle x_1,\ldots,x_6 \mid[x_1,x_2]=x_3,[x_1,x_3]=x_4,[x_1,x_4]=x_5,\\
&&{[x_2,x_3]}=x_5+\varepsilon x_6, {[x_2,x_5]}=x_6,[x_3,x_4]=x_6\rangle\mbox{ with }\varepsilon\in\F/(\relstarplus),
\end{eqnarray*}
form a complete and irredundant list of representatives of the isomorphism classes of step-1 descendants of $L_{5,6}$.
\item[(5/7)] Set $L_{5,7}=\langle x_1,\ldots,x_5\mid [x_1,x_2]=x_3,[x_1,x_3]=x_4,[x_1,x_4]=x_5\rangle$. If $\char \F\neq 2$, then 
the Lie algebra $L_{5,7}$ has three isomorphism classes of step-$1$ descendants namely
\begin{eqnarray*}
L_{6,16}&=&\langle x_1,\ldots,x_6\mid [x_1,x_2]=x_3, [x_1,x_3]=x_4, {[x_1,x_4]}=x_5,{[x_2,x_5]}=x_6, [x_3,x_4]=-x_6\rangle;\\
L_{6,17}&=&\langle x_1,\ldots,x_6\mid [x_1,x_2]=x_3, [x_1,x_3]=x_4, [x_1,x_4]=x_5, [x_1,x_5]=x_6,[x_2,x_3]=x_6\rangle;\\
L_{6,18}&=&\langle x_1,\ldots,x_6\mid [x_1,x_2]=x_3, [x_1,x_3]=x_4, [x_1,x_4]=x_5, [x_1,x_5]=x_6\rangle.
\end{eqnarray*}
If $\char \F=2$ then $L_{6,17}$, $L_{6,18}$ and the Lie algebras
\begin{eqnarray*}
L_{6,4}^{(2)}(\varepsilon)=\langle x_1,\ldots,x_6\mid [x_1,x_2]=x_3,[x_1,x_3]=x_4,[x_1,x_4]=x_5,\\
{[x_2,x_3]}=\varepsilon x_6,{[x_2,x_5]}=x_6,[x_3,x_4]=x_6\rangle,
\end{eqnarray*}
where $\varepsilon$ runs through the elements of $\F/(\relstarplus)$, form a complete and irredundant set of representatives of the isomorphism types of step-1 descendants of $L_{5,7}$. 
\item[(5/8)]Set $L_{5,8}=\langle x_1,\ldots,x_5\mid [x_1,x_2]=x_4,[x_1,x_3]=x_5\rangle$. If $\char \F\neq 2$, then the Lie algebras 
\begin{eqnarray*}
L_{6,20}&=&\langle x_1,\ldots,x_6\mid  [x_1,x_2]=x_4, [x_1,x_3]=x_5, [x_1,x_5]=x_6, [x_2,x_4]=x_6\rangle;\\
L_{6,19}(\varepsilon)&=&\langle x_1,\ldots,x_6\mid  [x_1,x_2]=x_4, [x_1,x_3]=x_5, [x_1,x_5]=x_6, [x_2,x_4]=x_6, [x_3,x_5]=\varepsilon x_6\rangle,\\
&& \mbox{ where $\varepsilon\in \F^*/(\relstar)$},
\end{eqnarray*}
form a complete and irredundant set of representatives of the isomorphism classes of step-1 descendants of $L_{5,8}$. If $\char \F=2$ then such a set of representatives is formed by the Lie algebras $L_{6,19}(\varepsilon)$ with $\varepsilon\in\F^*/(\relstar)$, $L_{6,20}$ and the Lie algebra
\[L_{6,5}^{(2)}=\langle x_1,\ldots,x_6\mid [x_1,x_2]=x_4,[x_1,x_3]=x_5,{[x_2,x_5]}=x_6,[x_3,x_4]=x_6\rangle.\]
\item[(5/9)] Set $L_{5,9}=\langle x_1,\ldots,x_5\mid [x_1,x_2]=x_3,[x_1,x_3]=x_4, [x_2,x_3]=x_5\rangle$. If $\char \F\neq 2$ then the Lie algebras
\[L_{6,21}(\varepsilon)=\langle x_1,\ldots,x_6\mid [x_1,x_2]=x_3, [x_1,x_3]=x_4, [x_1,x_4]=x_6, [x_2,x_3]=x_5, 
[x_2,x_5]= \varepsilon x_6\rangle,\]
where $\varepsilon$ runs through the elements of $\F^*/(\relstar)$ form a complete and irredundant set of representatives of the isomorphism classes of step-1 descendants of $L_{5,9}$. If $\char \F=2$ then such a set of representatives is formed by the Lie algebras $L_{6,21}(\varepsilon)$ with $\varepsilon\in \F^*/(\relstar)$ and by the Lie algebra
\[L_{6,6}^{(2)}=\langle x_1,\ldots,x_6\mid[x_1,x_2]=x_3, [x_1,x_3]=x_4, [x_1,x_5]=x_6, [x_2,x_3]=x_5, 
[x_2,x_4]=x_6\rangle.\]
\end{itemize}

\subsection*{Step-2 descendants of $4$-dimensional Lie algebras}
\begin{itemize}
\item[(4/1)] Let $L_{4,1}$  be the  abelian Lie algebra  of dimension $4$. If $\char \F\neq 2$ the following is a complete and irredundant list of the representatives of the isomorphism classes of the step-2 descendants of $L_{4,1}$:
\[ L_{6,22}(\varepsilon) =\langle x_1,\ldots,x_6\mid [x_1,x_2]=x_5, [x_1,x_3]=x_6, [x_2,x_4]=\varepsilon x_6,
[x_3,x_4]=x_5\rangle\]
where $\varepsilon\in \F/(\relstar)$. If $\char \F=2$, then such a list is formed by the Lie algebras $L_{6,22}(\nu)$ as above, where, in this case, $\nu \in \F/(\relstarplus)$, and by the Lie algebras
\[ L_{6,7}^{(2)}(\eta) =\langle x_1,\ldots,x_6\mid [x_1,x_2]=x_5, [x_1,x_3]=x_6, [x_2,x_4]=\eta x_6,
[x_3,x_4]=x_5+x_6\rangle\]
where $\eta\in\F/(\relpsi)$ 
\item[(4/2)] Set $L_{4,2}=\langle x_1,\ldots,x_4\mid [x_1,x_2]=x_3\rangle$. If $\char \F\neq 2$, then the following Lie algebras form a complete and irredundant set of representatives of the isomorphism classes of the step-2 descendants of $L_{4,2}$:
\begin{eqnarray*}
L_{6,27}&=&\langle x_1,\ldots,x_6\mid [x_1,x_2]=x_3, [x_1,x_3]=x_5, [x_2,x_4]= x_6 \rangle;\\
L_{6,23}&=&\langle x_1,\ldots,x_6\mid [x_1,x_2]=x_3, [x_1,x_3]=x_5, [x_1,x_4]=x_6, [x_2,x_4]= x_5\rangle;\\
L_{6,25}&=&\langle x_1,\ldots,x_6\mid [x_1,x_2]=x_3, [x_1,x_3]=x_5, [x_1,x_4]=x_6\rangle;\\
L_{6,24}(\varepsilon)&=&\langle x_1,\ldots,x_6\mid [x_1,x_2]=x_3, [x_1,x_3]=x_5, {[x_1,x_4]}=\varepsilon x_6,\\ 
&&{[x_2,x_3]}=x_6, [x_2,x_4]= x_5\rangle\mbox{ where $\varepsilon\in \F/(\relstar)$}.
\end{eqnarray*}
If $\char \F=2$ then such a set of representatives is formed by the Lie algebras $L_{6,27}$, $L_{6,23}$, $L_{6,24}(\nu)$, where 
$\nu \in \F/(\relstarplus)$, $L_{6,25}$ and by the Lie algebras
\begin{eqnarray*}
L_{6,8}^{(2)}(\eta)=\langle x_1,\ldots,x_6&\mid& [x_1,x_2]=x_3,[x_1,x_3]=x_5,[x_1,x_4]=\eta x_6,\\
&&{[x_2,x_3]}=x_6,[x_2,x_4]=x_5+x_6\rangle \mbox{ where $\eta\in \F/(\relpsi)$}; 
\end{eqnarray*}
\item[(4/3)] The Lie algebra $L_{4,3}=\langle x_1,\ldots,x_4\mid[x_1,x_2]=x_3,[x_1,x_3]=x_4\rangle$ has only one isomorphism class of step-$2$ descendants namely
\[L_{6,28}=\langle x_1,\ldots,x_6\mid [x_1,x_2]=x_3, [x_1,x_3]=x_4, [x_1,x_4]=x_5, [x_2,x_3]=x_6\rangle.\]
\end{itemize}

\subsection*{Step-3 descendants of $3$-dimensional Lie algebras}

\begin{itemize}
\item[(3/1)] The abelian Lie algebra $L_{3,1}$ has a unique isomorphism type of step-3 descendants, namely
\[L_{6,26}=\langle x_1,\ldots,x_6\ |\ [x_1,x_2]=x_4,\ [x_1,x_3]=x_5,\ 
[x_2,x_3]=x_6\rangle.\]
\item[(3/2)] The Lie algebra $L_{3,2}=\langle x_1,x_2,x_3 | [x_1,x_2]=x_3\rangle$ has no step-3 descendants.
\end{itemize}

By explicitly listing the isomorphism classes of 6-dimensional nilpotent 
Lie algebras, the following theorem gives a summary of the results stated in 
this section. 
Recall that in a field $\F$ of characteristic~2,
$\relpsi$ denotes the equivalence relation that corresponds to the coset partition of $\F$ with respect to the subspace 
$\{x^2+x\ |\ x\in\F\}$

\begin{thm}\label{classif}
\begin{enumerate}
\item[(I)] Over a field $\F$ of characteristic different from~$2$, the list 
of the 
isomorphism types of $6$-dimensional 
nilpotent Lie algebras is the following:
$L_{5,k}\oplus\F$ with $k\in\{1,\ldots,9\}$; 
$L_{6,k}$ with $k\in\{10,\ldots,18,20,23,25,\ldots,28\}$; $L_{6,k}(\varepsilon_1)$
with $k=\{19,21\}$ and $\varepsilon_1\in\F^*/(\relstar)$;
$L_{6,k}(\varepsilon_2)$ with $k\in\{22,24\}$ and 
$\varepsilon_2\in\F/(\relstar)$. 
\item[(II)] Over a field $\F$ of characteristic~$2$, 
the isomorphism types of $6$-dimensional 
nilpotent Lie algebras are $L_{5,k}\oplus\F$ 
with $k\in\{1,\ldots,9\}$;
$L_{6,k}$ with $k\in\{10,\ldots,13,15,17,18,20,23,25,\ldots,28\}$;
$L_{6,k}(\varepsilon_1)$ with $k=\{19,21\}$ and $\varepsilon_1\in\F^*/(\relstar)$;
$L_{6,k}(\varepsilon_2)$ with $k=\{22,24\}$ and $\varepsilon_2\in\F/(\relstarplus)$;
$L_{6,k}^{(2)}$ with $k=\{1,2,5,6\}$;
$L_{6,k}^{(2)}(\varepsilon_3)$ with 
$k=\{3,4\}$ and $\varepsilon_3\in\F/(\relstarplus)$;
$L_{6,k}^{(2)}(\varepsilon_4)$
with $k\in\{7,8\}$ and $\varepsilon_4\in\F/(\relpsi)$ 
\end{enumerate}
\end{thm}

Theorem~\ref{classif} follows from the statements concerning the 
descendants of this section. These statements are proved in 
Section~\ref{orbsect}. Noting that $\F/(\relstar)=
\F^*/(\relstar)\cup\{0\}$, the main Theorem~\ref{main} is a 
consequence of Theorem~\ref{classif}. If $\F$ is an algebraically closed field
or a perfect field of characteristic~2, then $s=1$. If $\F$ is a finite field
of size $q$ with $q$ odd then $s=2$. Further, if $\F$ is a perfect 
field of characteristic~2 then $t=1$. Finally, if $\F$ is a finite field of characteristic 2
then $u=2$, and if $\F$ is algebraically closed of characteristic 2 then $u=1$.
This gives that the number 
of isomorphism types of 6-dimensional nilpotent Lie algebras is 34 over 
a finite field of characteristic different from 2; is 30 over an algebraically
closed field of characteristic different from 2; it is 
36 over a finite field of characteristic~2; and it is 34 over an algebraically 
closed field of characteristic 2.

\section{Vector spaces with forms}\label{formsect}

As cocycles are alternating bilinear forms, and vector spaces with quadratic forms will play an important role in determining isomorphisms within certain descendants of $L_{4,1}$ and $L_{4,2}$, we summarize in this section some basic facts concerning quadratic and bilinear forms. Suppose that $V$ is a vector space over a field $\F$ and $f$ is a function from $V\times V$ to $\F$. 
The function $f$ is  said to be a {\em bilinear form} if $f$ is linear in both of its variables. Further, $f$ is said to be {\em symmetric} if $f(u,v)=f(v,u)$, $f$ is said to be alternating if $f(v,v)=0$, while $f$ is said to be {\em skew-symmetric} if $f(u,v)=-f(v,u)$ for all $v,\ u\in V$. An alternating form is always skew-symmetric, while the converse of this statement is only valid if $\char \F\neq 2$. In characteristic~2, there are skew-symmetric forms that are not alternating. For a map $Q:V\rightarrow\F$ define $f_Q:V\times V\rightarrow\F$ as 
\[f_Q(u,v)=Q(u+v)-Q(u)-Q(v).\]
Then $Q$ is said to be a {\em quadratic form} if $Q(\alpha v)=\alpha^2Q(v)$ for all $\alpha\in\F$ and $v\in V$ and $f_Q$ is a bilinear form. In this case the bilinear form $f$ is called the {\em associated bilinear form} of $Q$. If $f$ is a symmetric or skew-symmetric bilinear form on a vector space $V$ and $U\subseteq V$ then let $U^\perp$ denote the {\em orthogonal complement} of $U$ in $V$:
\[U^\perp=\{v\in V\ |\ f(u,v)=0\mbox{ for all }u\in U\}.\]
The {\em radical} of the form $f$ is defined as $V^\perp$ and $f$ is said to be {\em non-singular} if $V^\perp=0$; otherwise $f$ is said to be {\em singular}. If $U$ is a subspace of $V$ then an orthogonal form $Q$ or a bilinear form $f$ can be restricted to $U$ and the restriction is a form with the same symmetrical properties as $f$. Such a space $U$ is called {\em singular} or {\em non-singular}, if the restriction of the form to $U$ is singular or non-singular, respectively.\par 
If $V$ is a vector space with basis $\{b_1,\ldots,b_n\}$ then, for $i,\ j\in\{1,\ldots,n\}$ with $i\neq j$, let $\Delta_{i,j}$ denote the alternating bilinear form defined as $\Delta_{i,j}(b_i,b_j)=-\Delta_{i,j}(b_j,b_i)=1$ and $\Delta_{i,j}(b_k,b_l)=0$
otherwise. Then the set of forms $\Delta_{i,j}$ with $i<j$ is a basis for the linear space of alternating bilinear forms on $V$. 
Suppose that $V$ is a vector space with a bilinear form $f$. If $g\in\GL(V)$ then the form $gf$ defined by $(gf)(u,v)=f(gu,gv)$ is also a bilinear form on $V$. Further, $gf$ is alternating, skew-symmetric, or symmetric if and only if $f$ is alternating, skew-symmetric, or symmetric, respectively. This defines a $\GL(V)$-action on the set of bilinear forms on $V$. The following lemma is well-known, see for example~\cite[Theorem~8.10.1]{lang}.

\begin{lem}\label{altlemma}
Let $V$ be a vector space with a fixed basis $\{b_1,\ldots,b_n\}$ and set $n_1=n$ if $n$ is even, $n_1=n-1$ if $n$ is odd.
\begin{enumerate}
\item[(i)] If $\Delta$ is a non-singular alternating bilinear form on $V$ then $n$ is even.
\item[(ii)] The group $\GL(V)$ has $\lfloor n/2\rfloor$ orbits on the set of alternating  bilinear forms on $V$ with orbit representatives
\[\Delta_{1,2},\ \Delta_{1,2}+\Delta_{3,4},\ldots,\Delta_{1,2}+\Delta_{3,4}+\cdots +\Delta_{n_1-1,n_1}.\]
\end{enumerate}
\end{lem}

Let $f$ be a form of one of the types above on a vector space $V$ over a field $\F$. If $\{b_1,\ldots,b_n\}$ is a fixed basis of $V$ then the Gram 
matrix $\mathbb G$ is defined with respect to this basis as the matrix whose $(i,j)$ entry is $f(b_i,b_j)$. The Gram determinant will be crucial for separating isomorphism types within  parametric families of Lie algebras in characteristic different 
from~2. However, in characteristic~2, another invariant, namely the Arf invariant, will be needed. Let $Q$ be a quadratic form on a vector space $V$ over a field of characteristic~2 with associated bilinear form $f$. Note, in this case, that $f$ is alternating.
Assume that $f$ is non-singular, which implies by Lemma~\ref{altlemma}(i), that $\dim V$ is even. Let $e_1,\ldots,e_k,f_1,\ldots,f_k$ be a symplectic basis of $V$; that is $f(e_i,e_j)=f(f_i,f_j)=0$ and $f(e_i,f_j)=\delta_{i,j}$ where $\delta_{i,j}$ is
the Kronecker-delta. Define the {\em Arf invariant} $\delta_Q$ of $Q$ with respect to the given basis as \[q=\sum_{i=1}^kQ(e_i)Q(f_i).\]
Of course, the Arf invariant of $Q$ depends on the chosen symplectic basis of $V$. However, the following is valid.

\begin{lem}\label{gramlemma}
Let $V$ be a vector space over a field and let $Q$ be a quadratic form with non-singular associated bilinear form.
\begin{itemize}
\item[(i)] Let $\mathbb G_1$ and $\mathbb G_2$ denote the Gram matrices of $V$ with respect to two bases of $V$. Then $\det \mathbb G_1/\det\mathbb G_2$ is an element of the multiplicative subgroup $\{x^2\ |\ x\in\F^*\}$ of $\F^*$.
\item[(ii)]
Assume that $\char \F=2$ and suppose that $q_1$ and $q_2$ are the values of the Arf invariant of $Q$ with respect to two symplectic bases. Then $q_1+q_2$ is an element of the additive subgroup $\{x^2+x\ |\ x\in\F\}$ of $\F$. 
\end{itemize}
\end{lem}
\begin{proof}
Statement~(i) is well-known, see for instance equation~(8.1.8) in~\cite{cohn}. Statement~(ii) is proved in~\cite[Theorem 8.11.12]{cohn}.
\end{proof}

Suppose that $Q$ is a quadratic form on a vector space $V$ with associated bilinear form $f_Q$. A vector $v\in V$ is said to be {\em singular} if $Q(v)=0$, and  it is {\em isotropic} if $f(v,v)=0$. A subspace $U$ of $V$ is {\em totally singular} if $Q(u)=0$ for all $u\in U$, while $U$ is said to be totally isotropic if $f(u,u)=0$ for all $u\in U$. If $\char \F\neq 2$, then the notions of singular and isotropic, and those of totally singular  and totally isotropic can be freely interchanged. This, however, is no longer true if $\char \F=2$. Let $V$ be a vector space with a quadratic form $Q$ and let $G$ be a group acting on $V$. Then we say that $G$ {\em preserves $Q$ modulo scalars} if for each $g\in G$ there is some $\alpha_g\in\F$ such that
\[Q(gv)=\alpha_gQ(v).\]

\begin{lem}\label{arflemma}
Suppose that $V$ is a vector space over a field $\F$ with a quadratic form $Q$ whose associated bilinear form is $f$, and let $G$ be a subgroup of  $\GL(V)$ preserving $Q$ modulo scalars. Suppose that $S_1$ and $S_2$ are $2$-dimensional subspaces of $V$ in the same $G$-orbit, and let $\{b_1,b_2\}$ and $\{c_1,c_2\}$ be bases of $S_1$ and $S_2$, respectively.
\begin{enumerate}
\item[(i)] 
Suppose that $S_1$ and $S_2$ are non-singular subspaces, and let $\mathbb G_1$ and $\mathbb G_2$ be the Gram matrices of $S_1$ and $S_2$ with respect to the given bases. Then $(\det \mathbb G_1)/(\det \mathbb G_2)\in\{\alpha^2\ |\ \alpha\in\F^*\}$. 
\item[(ii)] Suppose that $\char \F=2$, that $f$ is non-singular, 
that the given bases of $S_i$ are symplectic, 
and that $q_1$ and $q_2$ are the Arf invariants of $S_1$ and $S_2$ with respect to the given bases. Then $q_1+q_2\in\{\alpha+\alpha^2\ |\ \alpha\in\F\}$. 
\item[(iii)]
Suppose that $\char \F=2$ and that $f$ is identically zero on $S_i$, for $i=1,2$, and set $q_1=Q(b_1)Q(b_2)$ and $q_2=Q(c_1)Q(c_2)$. Then there exist $\alpha\in\F^*$ and $\beta\in\F$ such that $q_2=\alpha^2q_1+\beta^2$. 
\end{enumerate}
\end{lem}
\begin{proof}
Suppose that $g\in G$ such that $gS_1=S_2$. Since $G$ preserves the form $Q$ modulo scalars, there is some $\alpha\in\F$ such that $Q(gv)=\alpha Q(v)$, and $f(gu,gv)=\alpha f(u,v)$ for all $u,\ v\in V$. Let us prove statement~(i) first. The elements $gb_1$ and $gb_2$ form a basis for $S_2$ and the Gram matrix of $S_2$ with respect to this basis is $\alpha \mathbb G_1$ with determinant $\alpha^2\det \mathbb G_1$.  By a remark above, there is some $\beta\in\F^*$ such that $\alpha^2\det \mathbb G_1=\beta^2\det \mathbb G_2$, and hence $(\det \mathbb G_1)/(\det \mathbb G_2)=(\beta/\alpha)^2$ as claimed.\par
Let us now prove the second assertion. Since $b_1,\ b_2$ is a symplectic basis of $S_1$, we obtain that $f(gb_1,gb_1)=f(gb_2,gb_2)=0$ and $f(gb_1,gb_2)=f(gb_2,gb_1)=\alpha$. Hence the basis $(1/\alpha)gb_1,\ gb_2$ is a symplectic
basis of $S_2$ and the Arf invariant with respect to this basis is
\[Q\left(\frac 1\alpha gb_1\right)Q(gb_2)=\frac{1}{\alpha^2}\alpha^2Q(b_1)Q(b_2)=q_1.\]
Thus Lemma~\ref{arflemma} gives that $q_1+q_2\in \{\alpha+\alpha^2\ |\ \alpha\in\F\}$.\par
(iii) First we note that if $f$ is identically zero, then $Q(\gamma b_1+\delta b_2)=\gamma^2Q(b_1)+\delta^2Q(b_2)$. Since $gS_1=S_2$, there are $\alpha_1,\ \alpha_2,\ \beta_1,\ \beta_2\in\F$ such that $c_1=g(\alpha_1b_1+\alpha_2 b_2)$ and $c_2=g(\beta_1b_1+\beta_2b_2)$. As $c_1$ and $c_2$ are linearly independent, we obtain that $\alpha_1\beta_2+\alpha_2\beta_1\neq 0$. 
Then
\begin{eqnarray*}
q_2&=&Q(c_1)Q(c_2)\\
&=&Q(g(\alpha_1b_1+\alpha_2 b_2))Q(g(\beta_1b_1+\beta_2b_2))\\
&=&\alpha^2(\alpha_1^2Q(b_1)+\alpha_2^2Q(b_2))(\beta_1^2Q(b_1)+\beta_2^2Q(b_2))\\
&=&\alpha^2(\alpha_1^2\beta_1^2Q(b_1)^2+(\alpha_1^2\beta_2^2+\alpha_2^2\beta_1^2)Q(b_1)Q(b_2)+\alpha_2^2\beta_2^2Q(b_2)^2)\\
&=&(\alpha\alpha_1\beta_1Q(b_1)+\alpha\alpha_2\beta_2Q(b_2))^2+(\alpha\alpha_1\beta_2+\alpha\alpha_2\beta_1)^2q_2.
\end{eqnarray*}
Since $\alpha\neq 0$ and $\alpha_1\beta_2+\alpha_2\beta_1\neq 0$, the assertion follows.
\end{proof}

We close this section with an interesting byproduct of our work. In the course of determining the isomorphism classes of the descendants of the 4-dimensional abelian Lie algebra, we determined, for an arbitrary 4-dimensional vector space $V$,  the $\GL(V)$-orbits on the 2-dimensional subspaces of $V\wedge V$. As this result may have applications elsewhere, we state it separately.

\begin{thm}\label{kleinth}
Let $V$ be a vector space of dimension~$4$ over a field $\F$ with basis $\{v_1,v_2,v_3,v_4\}$. If $\char \F\neq 2$ then set \[\mathcal P=\{\langle v_1\wedge v_2,v_1\wedge v_3\rangle\}\cup\{\langle v_1\wedge v_2+v_3\wedge v_4,v_1\wedge v_3+
\varepsilon v_2\wedge v_4\rangle\, |\,\varepsilon\in\F/(\relstar)\}.\]
If $\char \F=2$, then let 
\begin{eqnarray*}
\mathcal P&=&\{\langle v_1\wedge v_2,v_1\wedge v_3\rangle,\langle v_1\wedge v_2+v_3\wedge v_4,
\rangle,\\ 
&&\quad\langle v_1\wedge v_2+v_3\wedge v_4,v_1\wedge v_3+\omega v_2\wedge v_4+v_3\wedge v_4\rangle\mid \omega\in\F/(\relpsi)\}\\
&\cup&\{\langle v_1\wedge v_2+v_3\wedge v_4,v_1\wedge v_3+
\varepsilon v_2\wedge v_4\rangle\, |\,\varepsilon\in\F/(\relstarplus)\}.
\end{eqnarray*}
Then $\mathcal P$ is a complete and irredundant set of representatives of the $\GL(V)$-orbits on the two-dimensional subspaces of $V\wedge V$. 
\end{thm}

The proof of Theorem~\ref{kleinth} will be given in Section~\ref{orbsect}.

\section{The calculation of 6-dimensional descendants}\label{orbsect}

In this section we prove the results stated in Section~\ref{ressect}. In order to make the calculations more compact, we introduce uniform notation. In each of the sections, $L$ will denote a fixed nilpotent Lie algebra with dimension $d$, where $d<6$, for which the descendants will be computed. The algebra $L$ will be given by a multiplication table as in Section~\ref{ressect}. Recall that the forms $\Delta_{i,j}$ with $i\leq j$ form a basis for the space of alternating bilinear forms on $L$, and the cohomology spaces $Z^2(L,\F)$, $B^2(L,\F)$, and $H^2(L,\F)$ will be determined in terms of this basis. These spaces will be described for each of the Lie algebras without proof, as it is an easy exercise to compute them in these cases. If $\Delta$ is an element of $Z^2(L,\F)$, then
$\overline\Delta$ denotes its image in $H^2(L,\F)$.\par
The automorphism group of  $L$ will be described as a group of $(d\times d)$-matrices with respect to the given basis of $L$. The automorphism groups will also
be presented without proof as it is in general easy to verify that the given matrices do indeed form the full automorphism group. We use column notation to describe the automorphisms; that is, the $i$-th column of the given matrix will contain the image of the $i$-th basis vector of $L$. 
By Theorem~\ref{coctheorem}, 
we will need to compute the $\Aut(L)$-orbits on the $(6-d)$-dimensional allowable 
subspaces of $H^2(L,\F)$. The set of these subspaces will be denoted by $\S$. The quotient $H^2(L,\F)$ will be given with a fixed basis $\Gamma_1,\ldots,\Gamma_k$. The element $\alpha_1\Gamma_1+\cdots+\alpha_k\Gamma_k$ of $H^2(L,F)$ will be written simply as $(\alpha_1,\ldots,\alpha_k)$. The action of $\Aut(L)$ on $H^2(L,\F)$ will be computed explicitly.\par
Now we can start determining the 6-dimensional descendants of the nilpotent Lie algebras with dimension at most $5$.

\subsection*{$\mathbf L_{5,1}$, $\mathbf L_{5,4}$, $\mathbf L_{4,3}$, $\mathbf L_{3,1}$, $\mathbf L_{3,2}$}

First we compute the descendants in the easy cases. As there are no non-singular alternating bilinear forms on an odd-dimensional space (Lemma~\ref{altlemma}(i)), the Lie algebra $L_{5,1}$ has no step-1 descendants. If $L=L_{5,4}$ then $Z^2(L,F)=\langle \Delta_{1,2},\Delta_{1,3},\Delta_{1,4},\Delta_{2,3},\Delta_{2,4},\Delta_{3,4}\rangle$, $B^2(L,F)=\langle\Delta_{1,2}+\Delta_{3,4}\rangle$, and  $H^2(L,\F)=\langle \overline{\Delta_{1,3}},\overline{\Delta_{1,4}},
\overline{\Delta_{2,3}},\overline{\Delta_{2,4}},\overline{\Delta_{3,4}}\rangle$. Note that each element of $H^2(L,\F)$ has $x_5$ in its radical. Hence $\S=\emptyset$, and so there are no step-$1$ descendants of $L$.\par 
If $L=L_{4,3}$ then $H^2(L,\F)$ is 2-dimensional, spanned by $\overline{\Delta_{1,4}}$, and $\overline{\Delta_{2,3}}$. Since $H^2(L,\F)$ is allowable, $\mathcal S=H^2(L,\F)$ which gives that $L$ has a single isomorphism type of step-2 descendants, as claimed in 
Section~\ref{ressect}. Similarly, if $L=L_{3,1}$, we obtain that $H^2(L,\F)$ is 3-dimensional spanned by $\overline{\Delta_{1,2}}$,  $\overline{\Delta_{1,3}}$, and $\overline{\Delta_{2,3}}$. Since $H^2(L,\F)$ is allowable, we have that $\mathcal S=H^2(L,\F)$, which shows that $L$ has only one isomorphism class of step-3 descendants. Finally, if $L=L_{3,2}$, then $H^2$ is 2-dimensional spanned by 
$\overline{\Delta_{1,3}}$ and $\overline{\Delta_{2,3}}$ which shows, in this case, that $\mathcal S=\emptyset$.

\subsection*{$\mathbf L_{5,2}$}

Set $L=L_{5,2}$. It is routine to check that $\Aut(L)$ is the group of invertible matrices of the form 
\begin{equation}\label{A52}
A=\begin{pmatrix} 
a_{11} & a_{12} & 0 & 0 & 0 \\
a_{21} & a_{22} & 0 & 0 & 0\\
a_{31} & a_{32} & u & a_{34} & a_{35} \\
a_{41} & a_{42} & 0 & a_{44} & a_{45}\\
a_{51} & a_{52} & 0 & a_{54} & a_{55}
\end{pmatrix}
\end{equation}
with $u = a_{11}a_{22}-a_{12}a_{21}$. As 
$Z^2(L,\F)=\langle \Delta_{1,2},\Delta_{1,3},\Delta_{1,4},\Delta_{1,5},\Delta_{2,3},\Delta_{2,4},\Delta_{2,5},\Delta_{4,5}\rangle$ and $B^2(L,\F)=\langle \Delta_{1,2}\rangle$, we obtain that $H^2(L,\F)=\langle \overline{\Delta_{1,3}},\overline{\Delta_{1,4}},\overline{\Delta_{1,5}},\overline{\Delta_{2,3}},\overline{\Delta_{2,4}},\overline{\Delta_{2,5}},\overline{\Delta_{4,5}}\rangle$ and 
\[\mathcal S=\left\{\langle (a,b,c,d,e,f,g)\rangle\ |\ g\neq 0\mbox{ and }(a,d)\neq (0,0)\right\}.\]
As the matrix $A$ is invertible, $u\neq 0$. For $\vartheta=(a,b,c,d,e,f,g)\in H^2(L,\F)$, we compute that $A\vartheta =(\bar a,\bar b,\bar c,\bar d,\bar e,\bar f,\bar g)$ where 
\begin{eqnarray*}
\bar a & =& u(a_{11}a+a_{21}d);\\
\bar b & =& a_{11}a_{34}a+a_{11}a_{44}b+a_{11}a_{54}c+a_{21}a_{34}d+a_{21}a_{44}e+a_{21}a_{54}f +(a_{41}a_{54}-a_{44}a_{51})g; \\
\bar c & =& a_{11}a_{35}a+a_{11}a_{45}b+a_{11}a_{55}c+a_{21}a_{35}d+a_{21}a_{45}e+a_{21}a_{55}f +(a_{41}a_{55}-a_{45}a_{51})g; \\
\bar d & =& u(a_{12}a+a_{22}d);\\
\bar e & =& a_{12}a_{34}a+a_{12}a_{44}b+a_{12}a_{54}c+a_{22}a_{34}d+a_{22}a_{44}e+a_{22}a_{54}f +(a_{42}a_{54}-a_{44}a_{52})g; \\
\bar f & =& a_{12}a_{35}a+a_{12}a_{45}b+a_{12}a_{55}c+a_{22}a_{35}d+a_{22}a_{45}e+a_{22}a_{55}f +(a_{42}a_{55}-a_{45}a_{52})g; \\
\bar g & =& (a_{44}a_{55}-a_{45}a_{54})g.
\end{eqnarray*}
Choose $S=\langle (a,b,c,d,e,f,g)\rangle\in\S$ and let $B$ be the first of the following matrices if $a\neq 0$, or the second if $a=0$:
\[\begin{pmatrix}
1&-dg&0&0&0\\
0&ag&0&0&0\\
0&0&ag&-b&-c\\
0&-(af-cd)&0&a&0\\
0&ae-bd&0&0&a
\end{pmatrix}, \quad
\begin{pmatrix}
0&-dg&0&0&0\\
1&0&0&0&0\\
0&0&dg&-e&-f\\
0&cd&0&d&0\\
0&-bd&0&0&d
\end{pmatrix}.\]
Then $B S=\langle (1,0,0,0,0,0,1)\rangle$ and hence the group of matrices of the form~\eqref{A52}  has only one orbit on $\S$. Thus $L$ has only one step-$1$ descendant, namely $L_{6,10}$.

\subsection*{$\mathbf L_{5,3}$}

Set $L=L_{5,3}$. The group $\Aut(L)$ consists of the invertible matrices of the form
\begin{equation}\label{A53}
A=\begin{pmatrix} a_{11} & 0 & 0 & 0 & 0 \\
a_{21} & a_{22} & 0 & 0 & 0\\
a_{31} & a_{32} & a_{11}a_{22} & 0 & 0 \\
a_{41} & a_{42} & a_{11}a_{32} & a_{11}^2a_{22} & a_{45}\\
a_{51} & a_{52} & 0 & 0 & a_{55}\end{pmatrix}.
\end{equation}
We have that $Z^2(L,\F)=\langle\Delta_{1,2},\Delta_{1,3},\Delta_{1,4},\Delta_{1,5},\Delta_{2,3},\Delta_{2,5}\rangle$, $B^2(L,\F)=\langle \Delta_{1,2},\Delta_{1,3}\rangle$, and so $H^2(L,\F)=\langle \overline{\Delta_{1,4}},\overline{\Delta_{1,5}},\overline{\Delta_{2,3}},\overline{\Delta_{2,5}}\rangle$. Further,
\[\S=\{\langle (a,b,c,d)\rangle\mid a,d\neq 0\}.\]
If $\vartheta=(a,b,c,d)\in H^2(L,\F)$, then $A\vartheta =(\bar a,\bar b,\bar c,\bar d)$ where 
\[\bar a  = a_{11}^3a_{22}a,\quad
\bar b = a_{11}a_{45}a+a_{11}a_{55}b+a_{21}a_{55}d,\quad
\bar c  = a_{11}a_{22}^2c,\quad
\bar d  = a_{22}a_{55}d.\]
Choose $S=\langle(a,b,c,d)\rangle\in\S$. Set $c_1=1$ if $c=0$ and set $c_1=c/(ad^2)$ otherwise. Let $B$ denote the matrix
\[\begin{pmatrix}
c_1&0&0&0&0\\
0&c_1&0&0&0\\
0 &0&c_1^2&0&0\\
0&0&0&c_1^3&0\\
0&0&0&0&c_1^3
\end{pmatrix}\begin{pmatrix}
d&0&0&0&0\\
0&1&0&0&0\\
0&0&d&0&0\\
0&0&0&d^2&-bd^2\\
0&0&0&0&ad^2
\end{pmatrix}.\]
Easy computation shows that $BS=\langle (1,0,0,1)\rangle$ if $c=0$ while $BS=\langle (1,0,1,1)\rangle$ otherwise. Hence the group of matrices of the form~\eqref{A53} has two orbits on $\mathcal S$, namely $\langle(1,0,0,1)\rangle$ and $\langle(1,0,1,1)\rangle$. The corresponding Lie algebras are $L_{6,12}$ and $L_{6,11}$.\par
The derived subalgebra $\langle x_3, x_4, x_6\rangle$ of both $L_{6,11}$ and $L_{6,12}$ is 3-dimensional. However, in $L_{6,11}$ the centralizer $\langle x_3,x_4,x_5,x_6\rangle$ of the derived subalgebra is $4$-dimensional, while in $L_{6,12}$ the centralizer $\langle x_2,x_3,x_4,x_5,x_6\rangle$ of $L_{6,12}'$ is 5-dimensional. Hence $L_{6,11}$ and $L_{6,12}$ are not isomorphic.

\subsection*{$\mathbf L_{5,5}$}
Let $L=L_{5,5}$. Invertible matrices of the form  
\begin{equation}\label{A55}
A=\begin{pmatrix}
 a_{11} & 0&0&0&0\\
a_{21}&a_{22}&0&0&0\\
a_{31}&a_{32}&a_{11}a_{22}& -a_{11}a_{21}&0\\
a_{41}&a_{42}&0&a_{11}^2&0\\
a_{51}&a_{52}&u&a_{54}&a_{11}^2a_{22},
\end{pmatrix}
\end{equation}
where $u=a_{11}a_{32}+a_{21}a_{42}-a_{22}a_{41}$, form the group $\Aut(L)$. We have, in addition, that
$Z^2(L,\F)=\langle \Delta_{1,2},\Delta_{1,3},\Delta_{1,4},\Delta_{1,5}+\Delta_{3,4},\Delta_{2,3},\Delta_{2,4}\rangle$,
$B^2(L,\F)=\langle \Delta_{1,2},\Delta_{1,3}+\Delta_{2,4}\rangle$, and so
$H^2(L,\F)=\langle \overline{\Delta_{1,4}},\overline{\Delta_{1,5}}+\overline{\Delta_{3,4}},\overline{\Delta_{2,3}},\overline{\Delta_{2,4}}\rangle$. 
The set of allowable subspaces of $H^2(L,\F)$ is 
\[\S=\{\langle(a,b,c,d)\rangle\mid b\neq 0\}.\] 
If $\vartheta=(a,b,c,d)\in H^2(L,\F)$, then $A\vartheta=(\bar a,\bar b,\bar c,\bar d)$ where
\begin{eqnarray*}
\bar a=a_{11}^3a+(a_{11}a_{54}+a_{11}^2a_{31}+a_{11}a_{21}a_{41})b-a_{11}a_{21}^2c+a_{11}^2a_{21}d;\\
\bar b= a_{11}^3a_{22}b,\quad \bar c=-a_{11}a_{22}a_{42}b+a_{11}a_{22}^2c,\quad \bar d=2a_{11}a_{22}a_{41}b-2a_{11}a_{21}a_{22}c+a_{11}^2a_{22}d.
\end{eqnarray*}
Choose $S=\langle(a,b,c,d)\rangle\in \S$. If $\char \F\neq 2$ then
\[\begin{pmatrix}
2b&0&0&0&0\\
0&b&0&0&0\\
-2a &0&2b^2&0&0\\
-d&c&0&4b^2&0\\
0&0&bd&0&4b^3
\end{pmatrix}S=\langle (0,1,0,0)\rangle,\]
which shows that the group of matrices of the form~\eqref{A55} is transitive on $\mathcal S$, and that $L$ has only one isomorphism type of step-1 descendants, namely $L_{6,13}$, as claimed.\par
Suppose now that $\char \F=2$. Set $b_1=d/b^2$ if $b\neq 0$ and set $b_1=1$ otherwise. Let $B$ denote the matrix 
\[\begin{pmatrix}
b_1&0&0&0&0\\
0&1&0&0&0\\
0 &0&b_1&0&0\\
0&0&0&b_1^2&0\\
0&0&0&0&b_1^2
\end{pmatrix}
\begin{pmatrix}
b&0&0&0&0\\
0&b&0&0&0\\
a &0&b^2&0&0\\
0&c&0&b^2&0\\
0&0&0&0&b^3
\end{pmatrix}. \]
Then $BS=\langle (0,1,0,1)\rangle$ if $d\neq 0$ and $BS=\langle (0,1,0,0)\rangle$ otherwise. Hence, in this case, the group of matrices of the form~\eqref{A55} has two orbits on $\mathcal S$, namely $\langle (0,1,0,0)\rangle$ and $\langle(0,1,0,1)\rangle$. The corresponding Lie algebras are $L_{6,13}$, and $L_{6,1}^{(2)}$, respectively.\par
We claim, for $\char \F= 2$ that  $L_{6,1}^{(2)}$ and $L_{6,13}$ are not isomorphic. It suffices to show that the allowable subspaces $\langle (0,1,0,0)\rangle$ and $\langle(0,1,0,1)\rangle$ are in different orbits under $\Aut(L)$. Suppose by contradiction that $A$ is of the form~\eqref{A55} such that $A\langle(0,1,0,1)\rangle=\langle(0,1,0,0)\rangle$. Then, as $\char \F=2$, the expression for $\bar d$ gives that $0  = a_{11}^2a_{22}$, which is impossible as $A$ is invertible. Hence the two descendants are non-isomorphic.

\subsection*{$\mathbf L_{5,6}$}

Set $L=L_{5,6}$. The group $\Aut(L)$ consists of the invertible matrices of the form   
\[A= \begin{pmatrix} a_{11} & 0 & 0 & 0 & 0 \\
a_{21} & a_{11}^2 & 0 & 0 & 0\\
a_{31} & a_{32} & a_{11}^3 & 0 & 0 \\
a_{41} & a_{42} & a_{11}a_{32} & a_{11}^4 & 0 \\
a_{51} & a_{52} & u & v & a_{11}^5
\end{pmatrix},\]
where $u=-a_{11}^2a_{31}+a_{11}a_{42}+a_{21}a_{32}$ and $v=a_{11}^3a_{21}+a_{11}^2a_{32}$. As $Z^2(L,\F)=\langle \Delta_{1,2},\Delta_{1,3},\Delta_{1,4},\Delta_{1,5}+\Delta_{2,4},\Delta_{2,3},\Delta_{2,5}-\Delta_{3,4}\rangle$ and $B^2(L,\F)=\langle
\Delta_{1,2},\Delta_{1,3},\Delta_{1,4}+\Delta_{2,3}\rangle$, we obtain that 
  $H^2(L,\F)=\langle \overline{\Delta_{1,5}}+
\overline{\Delta_{2,4}},\overline{\Delta_{2,3}},\overline{\Delta_{2,5}}-\overline{\Delta_{3,4}}\rangle$. Moreover,
\[\S=\{\langle(a,b,c)\rangle\mid (a,c)\neq (0,0)\}.\] 
If $\vartheta=(a,b,c)\in Z^2(L,\F)$, then $A\vartheta=(\bar a,\bar b,\bar c)$ where
\[\bar a = a_{11}^5(a_{11}a+a_{21}c),\quad \bar b = -2a_{11}^4a_{21}a+a_{11}^5b+(2a_{11}^3a_{42}-a_{11}^3a_{21}^2-a_{11}a_{32}^2)c,\quad\bar c = a_{11}^7 c.\]
Choose $S=\langle(a,b,c)\rangle\in \S$. Suppose first that $\char \F\neq 2$. Let $B$ denote the first of the following matrix if $c\neq 0$, or the second if $c=0$:
\[\begin{pmatrix}
2c&0&0&0&0\\
-2a&4c^2&0&0&0\\
0 &0&8c^3&0&0\\
0&-2(a^2+bc)&0&16c^4&0\\
0&0&-4c(a^2+bc)&-16ac^3&32c^5
\end{pmatrix},\quad
\begin{pmatrix}
2a&0&0&0&0\\
b&4a^2&0&0&0\\
0 &0&8a^3&0&0\\
0&0&0&16a^4&0\\
0&0&0&8a^3b&32a^5
\end{pmatrix}.\] 
Then, if $c\neq 0$, then $BS=\langle (0,0,1)\rangle$, while $BS=\langle (1,0,0)\rangle$ otherwise. This shows that $\Aut(L)$ has at most two orbits on $\mathcal S$, namely $\langle(0,0,1)\rangle$ and $\langle(1,0,0)\rangle$ and the corresponding Lie algebras are $L_{6,14}$ and $L_{6,15}$, respectively. The Lie algebras $L_{6,14}$ and $L_{6,15}$ are clearly non-isomorphic, as $(L_{6,15})'$ is abelian, while $(L_{6,14})'$ is not. \par
Assume now that $\char \F=2$. If $b=c=0$ then $S=\langle (1,0,0)\rangle$,
and so we may suppose that $(b,c)\neq(0,0)$. Let $B$ be the first of the following matrices if  $c=0$  and $b\neq 0$; while 
if $c\neq 0$ then let $B$ be the second matrix:
\[\begin{pmatrix}
b&0&0&0&0\\
0& b^2&0&0&0\\
0 &0& b^3&0&0\\
0&0&0& b^4&0\\
0&0&0&0& b^5
\end{pmatrix},\quad
\begin{pmatrix}
c&0&0&0&0\\
a&c^2&0&0&0\\
0 &ac&c^3&0&0\\
0&0&ac^2&c^4&0\\
0&0&a^2c&0&c^5
\end{pmatrix}.\]
Now, if $c=0$ and $b\neq 0$ then $BS=\langle (1,1,0)\rangle$, while $BS=\langle (0,b/c^3,1)\rangle$ otherwise. Thus the set of the subspaces $\langle (1,0,0)\rangle$, $\langle (1,1,0)\rangle$ and $\langle (0,\varepsilon,1)\rangle$ with $\varepsilon\in\F$ contains a representative of each of the $\Aut(L)$-orbits on $\mathcal S$. The Lie algebras corresponding to these subspaces are $L_{6,15}$, $L_{6,2}^{(2)}$ and $L_{6,3}^{(2)}(\varepsilon)$. \par
Let us now determine the possible isomorphisms among the algebras $L_{6,15}$, $L_{6,2}^{(2)}$ and $L_{6,3}^{(2)}(\varepsilon)$. First note that the derived subalgebras of $L_{6,15}$ and $L_{6,2}^{(2)}$ are abelian, while that of $L_{6,3}^{(2)}(\varepsilon)$ is not. Thus $L_{6,3}^{(2)}(\varepsilon)\not\cong L_{6,15}$ and  $L_{6,3}^{(2)}(\varepsilon)\not\cong L_{6,2}^{(2)}$. If $L_{6,2}^{(2)} \cong L_{6,15}$, then there is some $A\in\Aut(L)$ such that $A\langle(1,1,0)\rangle=\langle(1,0,0)\rangle$. However, the expression for $\bar b$ implies in this case that $a_{11}=0$, which makes $A$ non-invertible. Thus $L_{6,2}^{(2)} \not \cong L_{6,15}$.\par
Therefore it remains to determine the isomorphisms among the algebras $L_{6,3}^{(2)}(\varepsilon)$ with different values of $\varepsilon$. We claim that $L_{6,3}^{(2)}(\varepsilon)\cong L_{6,3}^{(2)}(\nu)$ if and only of $\varepsilon\relstarplus\nu$ where $\relstarplus$ is the equivalence relation defined at the beginning of Section~\ref{ressect}. To prove one direction of this claim, assume that $L_{6,3}^{(2)}(\varepsilon)\cong L_{6,3}^{(2)}(\nu)$. Then there is $A\in\Aut(L)$ such that $A\langle(0,\varepsilon,1)\rangle=\langle(0,\nu,1)\rangle$. The equations for $\bar b$ and $\bar c$ give that $a_{11}^5\varepsilon +a_{11}^3a_{21}^2+a_{11}a_{32}^2=a_{11}^7\nu$. Since $a_{11}\neq 0$, we may divide both sides of this equation by $a_{11}^5$ and obtain that $\varepsilon\relstarplus\nu$, as required. Now suppose, conversely, that $\varepsilon\relstarplus\nu$; that is, there are $\alpha\in\F^\ast$ and $\beta\in\F$ such that $\nu=\alpha^2\varepsilon+\beta^2$. Then
\[\begin{pmatrix}
\alpha^{-1}&0&0&0&0\\
0&\alpha^{-2}&0&0&0\\
0 &\beta\alpha^{-3}&\alpha^{-3}&0&0\\
0&0&\beta\alpha^{-4}&\alpha^{-4}&0\\
0&0&0&\beta\alpha^{-5}&\alpha^{-5}
\end{pmatrix}\langle(0,\varepsilon,1)\rangle=\langle(0,\nu,1)\rangle.\]
This proves the claim about the isomorphisms among the algebras $L_{6,3}^{(2)}(\nu)$.

\subsection*{$\mathbf L_{5,7}$}

The group $\Aut(L)$ consists of the invertible matrices of the form   
\[A= \begin{pmatrix} a_{11} & 0 & 0 & 0 & 0 \\
a_{21} & a_{22} & 0 & 0 & 0\\
a_{31} & a_{32} & a_{11}a_{22} & 0 & 0 \\
a_{41} & a_{42} & a_{11}a_{32} & a_{11}^2a_{22} & 0 \\
a_{51} & a_{52} & a_{11}a_{42} & a_{11}^2a_{32} & a_{11}^3a_{22}
\end{pmatrix}.\]
We have $Z^2(L,\F)=\langle\Delta_{1,2},\Delta_{1,3},\Delta_{1,4},\Delta_{1,5},\Delta_{2,3},\Delta_{2,5}-\Delta_{3,4}\rangle$ and $B^2(L,\F)=\langle\Delta_{1,2},\Delta_{1,3},\Delta_{1,4}\rangle$, and so  $H^2(L,\F)=\langle \overline{\Delta_{1,5}},\overline{\Delta_{2,3}},\overline{\Delta_{2,5}}-\overline{\Delta_{3,4}}\rangle$. Moreover,
\[\S=\{\langle (a,b,c)\rangle\mid (a,c)\neq (0,0)\}.\]
If $\vartheta=(a,b,c)\in H^2(L,\F)$, then $A\vartheta =(\bar a,\bar b,\bar c)$ where 
\[\bar a  = a_{11}^3a_{22}(a_{11}a+a_{21}c),\quad
\bar b  = a_{11}(a_{22}^2b + (2a_{22}a_{42}-a_{32}^2)c),\quad
\bar c  = a_{11}^3a_{22}^2 c.\]
Choose $S=\langle(a,b,c)\rangle\in\S$. If $b=c=0$ then $S=\langle (1,0,0)\rangle$. Let $B$ be the first, the second, or the third of the following matrices, in the cases when $c=0$ and $b\neq 0$; $c\neq 0$ and $\char \F\neq 2$; or $c\neq0$ and $\char \F=2$; respectively:
\[\begin{pmatrix}
b&0&0&0&0\\
0& b^2&0&0&0\\
0 &0& b^3&0&0\\
0&0&0& b^4&0\\
0&0&0&0& b^5
\end{pmatrix},\quad
\begin{pmatrix}
c&0&0&0&0\\
-a&2c&0&0&0\\
0 &0&2c^2&0&0\\
0&-b&0&2c^3&0\\
0&0&-bc&0&2c^4
\end{pmatrix},\quad 
\begin{pmatrix}
c&0&0&0&0\\
a&1&0&0&0\\
0 &0&c&0&0\\
0&0&0&c^2&0\\
0&0&0&0&c^3
\end{pmatrix}.\]
Then we obtain that $BS=\langle(1,1,0)\rangle$ if $c=0$ and $b\neq 0$; $BS=\langle(0,0,1)\rangle$ if $c\neq 0$ and $\char \F\neq 2$; while $BS=\langle(0,\varepsilon,1)\rangle$ if $c\neq 0$ and $\char \F=2$. Hence if $\char \F\neq 2$ then $\Aut(L)$ has at most three 
orbits on $\mathcal S$, namely $\langle (1,0,0)\rangle$, $\langle (1,1,0)\rangle$, $\langle (0,0,1)\rangle$ and the corresponding Lie algebras are $L_{6,18}$, $L_{6,17}$, and $L_{6,16}$, respectively.
If $\char \F=2$ then the 1-spaces $\langle (1,0,0)\rangle$, $\langle (1,1,0)\rangle$, $\langle (0,\varepsilon,1)\rangle$ contain
a set of representatives for the $\Aut(L)$-orbits on $\mathcal S$ with corresponding Lie algebras $L_{6,18}$, $L_{6,17}$ and $L_{6,4}^{(2)}(\varepsilon)$, respectively.\par
Note that the derived subalgebras $L_{6,17}$ and $L_{6,18}$ are abelian, while those of  $L_{6,16}$ and $L_{6,4}^{(2)}(\varepsilon)$ are not. Hence $L_{6,16}\not\cong L_{6,17}$, $L_{6,16}\not\cong L_{6,18}$, $L_{6,4}^{(2)}(\varepsilon) \not\cong L_{6,17}$ and $L_{6,4}^{(2)}(\varepsilon)\not\cong L_{6,18}$. Further, the centralizer $\langle x_3,x_4,x_5,x_6\rangle$ of $(L_{6,17})'$ is 4-dimensional while the centralizer $\langle x_2,x_3,x_4,x_5,x_6\rangle$ of $(L_{6,18})'$ is 5-dimensional, and hence $L_{6,17}\not\cong L_{6,18}$. Thus we are left with having to determine the possible isomorphisms between the algebras $L_{6,4}^{(2)}(\varepsilon)$ with different values of $\varepsilon$. We claim that $L_{6,4}^{(2)}(\varepsilon)\cong L_{6,4}^{(2)}(\nu)$ if and only if $\varepsilon\relstarplus\nu$. To show this claim, suppose first that $L_{6,4}^{(2)}(\varepsilon)\cong L_{6,4}^{(2)}(\nu)$. Then there is $A\in\Aut(L)$ such that $A\langle(0,\varepsilon,1)\rangle=\langle(0,\nu,1)\rangle$. The equations for $\bar b$ and $\bar c$ and the fact that $a_{11}\neq 0$ imply  that $a_{22}^2\varepsilon+a_{32}^2=a_{11}^2a_{22}^2\nu$, which, in turn, gives that $\varepsilon\relstarplus\nu$. Conversely, let us assume that $\varepsilon\relstarplus\nu$. Then there is some $\alpha\in\F^\ast$ and $\beta\in\F$ such that $\nu=\alpha^2\varepsilon+\beta^2$. Then 
\[\begin{pmatrix}
\alpha^{-1} & 0 & 0 & 0 & 0\\ 
0 & \alpha^{-1}\beta^{-1} & 0 & 0 & 0 \\
0 & \alpha^{-2} & \alpha^{-2}\beta^{-1}  & 0 & 0\\
0 & 0 & \alpha^{-3} & \alpha^{-3}\beta^{-1} & 0\\
0 & 0 & 0 & \alpha^{-4} & \alpha^{-4}\beta^{-1}
\end{pmatrix}\langle(0,\varepsilon,1)\rangle =\langle(0,\nu,1)\rangle.\]
This completes the proof of the claim concerning the
isomorphism among the $L_{6,4}^{(2)}(\varepsilon)$.

\subsection*{$\mathbf L_{5,8}$}

The group $\Aut(L)$ consists of the invertible matrices of the form   
\[A= \begin{pmatrix} a_{11} & 0 & 0 & 0 & 0 \\
a_{21} & a_{22} & a_{23} & 0 & 0\\
a_{31} & a_{32} & a_{33} & 0 & 0 \\
a_{41} & a_{42} & a_{43} & a_{11}a_{22} & a_{11}a_{23} \\
a_{51} & a_{52} & a_{53} & a_{11}a_{32} & a_{11}a_{33}
\end{pmatrix}.\]
We have $Z^2(L,\F)=\langle\Delta_{1,2},\Delta_{1,3},\Delta_{1,4},\Delta_{1,5},\Delta_{2,3},\Delta_{2,4},\Delta_{2,5}+\Delta_{3,4},\Delta_{3,5}\rangle$ and $B^2(L,\F)=\langle\Delta_{1,2},\Delta_{1,3}\rangle$, and so $H^2(L,\F)=\langle \overline{\Delta_{1,4}},\overline{\Delta_{1,5}},\overline{\Delta_{2,3}},\overline{\Delta_{2,4}},\overline{\Delta_{2,5}}+\overline{\Delta_{3,4}},\overline{\Delta_{3,5}}\rangle$. Further, 
\[\S=\left\{\langle (a,b,c,d,e,f)\rangle\mid \text{rank}\begin{pmatrix}a&b\\d&e\\e&f\end{pmatrix}=2\right\}.\]
If $\vartheta=(a,b,c,d,e,f)\in H$, then $A\vartheta =(\bar a,\bar b,\bar c,\bar d,\bar e,\bar f)$ where 
\begin{eqnarray*}
\bar a & =& a_{11}^2a_{22}a+a_{11}^2a_{32}b+a_{11}a_{21}a_{22}d+a_{11}(a_{21}a_{32}+a_{22}a_{31})e+a_{11}a_{31}a_{32}f;\\
\bar b & =& a_{11}^2a_{23}a+a_{11}^2a_{33}b+a_{11}a_{21}a_{23}d+a_{11}(a_{21}a_{33}+a_{23}a_{31})e +a_{11}a_{31}a_{33}f;\\
\bar c & =& (a_{22}a_{33}-a_{23}a_{32})c+(a_{22}a_{43}-a_{23}a_{42})d+(a_{22}a_{53}-a_{23}a_{52}+a_{32}a_{43}-a_{33}a_{42})e;\\
&+&(a_{32}a_{53}-a_{33}a_{52})f;\\
\bar d & =& a_{11}a_{22}^2 d +2a_{11}a_{22}a_{32}e +a_{11} a_{32}^2f;\\
\bar e & =& a_{11}a_{22}a_{23}d + a_{11}(a_{22}a_{33}+a_{23}a_{32})e+a_{11}a_{32}a_{33}f;\\
\bar f & =& a_{11}a_{23}^2 d + 2a_{11}a_{23}a_{33} e + a_{11}a_{33}^2 f. 
\end{eqnarray*}
Choose $S=\langle(a,b,c,d,e,f)\rangle\in\S$ and set $\delta_1=ae-bd$, $\delta_2=af-be$, $\delta_3=df-e^2$. Suppose first that $d\neq 0$. If $\delta_1\neq 0$, then let $B$ be the first of the following two automorphisms; if $\delta_1=0$ (which implies that $\delta_3\neq 0$), then let $B$ be the second:
\[\begin{pmatrix}
1&0&0&0&0\\
0&\delta_1&0&0&0\\
0&0&\delta_1&0&0\\
0&0&0&\delta_1&0\\
0&0&0&0&\delta_1
\end{pmatrix}\begin{pmatrix}
d&0&0&0&0\\
-a&1&e&0&0\\
0 &0&-d&0&0\\
0&0&c&d&de\\
0&0&0&0&-d^2
\end{pmatrix},\quad 
\begin{pmatrix}
1&0&0&0&0\\
0&\delta_3&0&0&0\\
1&0&\delta_3&0&0\\
0&0&0&\delta_3&0\\
0&0&0&0&\delta_3
\end{pmatrix}
\begin{pmatrix}
d&0&0&0&0\\
-a&1&e&0&0\\
0 &0&-d&0&0\\
0&0&c&d&de\\
0&0&0&0&-d^2
\end{pmatrix}.\]
Then $BS=\langle(0,1,0,1,0,\delta_3)\rangle$. Suppose now that $d=0$ and consider the following two cases: $e=0$ which implies that $a,\ f\neq 0$; and $e\neq 0$. In these cases, let $B$ denote the first or the second of the following transformations, respectively:
\[\begin{pmatrix}
1&0&0&0&0\\
0&-b&af&0&0\\
0 &a&0&0&0\\
0&0&0&-b&af\\
0&0&ac&a&0
\end{pmatrix},\quad
\begin{pmatrix}
e^2&0&0&0&0\\
af-be&1&0&0&0\\
-ae&0&e&0&0\\
0&c&0&e^2&0\\
0&0&0&0&e^3
\end{pmatrix}.\]
Then, in the first case, $BS=\langle (0,1,0,1,0,0)\rangle$, while in the second case, $BS=\langle (0,0,0,0,1,f)\rangle$. Thus if $d=e=0$ then we obtain that $S$ is in the orbit of $\langle (0,1,0,1,0,0)\rangle$. Therefore we may assume that $d=0$ and $e \neq 0$, and that $S$ is in the orbit of $\langle (0,0,0,0,1,f)\rangle$. If $f\neq 0$, then let $B_1$ denote the first of the following transformations; if $f=0$ and $\char \F\neq 2$ then let $B_1$ be the second:
\[\begin{pmatrix}
1&0&0&0&0\\
0&f&0&0&0\\
-1&-1&1&0&0\\
0&0&0&f&0\\
0&0&0&-1&1
\end{pmatrix},\quad \begin{pmatrix}
1&0&0&0&0\\
1&1&-1&0&0\\
-1&1&1&0&0\\
0&0&0&1&-1\\
0&0&0&1&1
\end{pmatrix}.\]
We obtain, in both cases, that $B_1\langle(0,0,0,0,1,f\rangle)=\langle(0,1,0,1,0,-1)\rangle$.\par
To summarize, in characteristic different from~2,  the set of subspaces $\langle (0,1,0,1,0,\varepsilon)\rangle$ contain a representative from each of the $\Aut(L)$-orbits on $\mathcal S$. In characteristic 2, these orbits are covered by the subspaces $\langle (0,1,0,1,0,\varepsilon)\rangle$ and  $\langle (0,0,0,0,1,0)\rangle$. The Lie algebras corresponding to these subspaces are $L_{6,19}(\varepsilon)$ and $L_{6,5}^{(2)}$. The Lie algebra $L_{6,19}(0)$ is written 
as $L_{6,20}$ in Section~\ref{ressect}, 
to minimize the difference between our list and that in~\cite{artw}.
The expression for $\bar d$ shows, in characteristic~2, 
that the
vectors in the $\Aut(L)$-orbit of $(0,0,0,0,1,0)$ all have 4-th 
coordinate 0, and so $\langle (0,1,0,1,0,\varepsilon)\rangle$ 
and $\langle (0,0,0,0,1,0)\rangle$ are indeed in different orbits.
Thus we only need to verify the isomorphisms
among the algebras $L_{6,19}(\varepsilon)$ with different values of $\varepsilon$. We claim that $L_{6,19}(\varepsilon)\cong L_{6,19}(\nu)$ if and only if $\varepsilon\relstar\nu$. To prove one direction of this claim assume that $L_{6,19}(\varepsilon)\cong L_{6,19}(\nu)$. Then there is some $A\in\Aut(L)$ such that $\langle (0,1,0,1,0,\varepsilon)\rangle=\langle (0,1,0,1,0,\nu)\rangle$. 
Considering the equations  for $\bar a$, $\bar b$ and $\bar f$ we obtain that 
\begin{equation}\label{eq0}
a_{22}^2+\varepsilon a_{32}^2\neq 0
\end{equation} 
and that
\begin{eqnarray}
a_{22}a_{23}+\varepsilon a_{32}a_{33}&=&0\label{eq1};\\
a_{23}^2+\varepsilon a_{33}^2-\nu a_{22}^2-\varepsilon\nu a_{32}^2
\label{eq2}&=&0.
\end{eqnarray}
If $\varepsilon=0$ then~\eqref{eq0} implies that $a_{22}\neq 0$, and equations~\eqref{eq1}--\eqref{eq2} give that $a_{23}=0$ and that $\nu=0$. Thus $L_{6,19}(0)$ is not isomorphic to $L_{6,19}(\nu)$ with $\nu\neq 0$. Therefore we may assume without loss of generality that $\varepsilon,\ \nu\neq 0$. Set $\delta=a_{22}^2+\varepsilon a_{32}^2$. Then routine computation shows that \[(a_{23}a_{22}-\varepsilon a_{23}a_{22})(a_{22}a_{23}+\varepsilon a_{32}a_{33})+(\varepsilon a_{32}^2-\delta)(a_{23}^2+\varepsilon a_{33}^2-\nu a_{22}^2-\varepsilon\nu a_{32}^2)=\delta(\nu a_{22}^2-\varepsilon a_{33}^2).\]
Since $\delta\neq 0$, equations~\eqref{eq1} and~\eqref{eq2} imply that $\nu a_{22}^2-\varepsilon a_{33}^2$. Thus either $a_{22}=a_{33}=0$ or $\varepsilon\relstar\nu$, as required. Suppose that $a_{22}=a_{33}=0$. Then $a_{23}^2=\varepsilon\nu a_{32}^2$. Since the matrix $A$ is invertible, we obtain that $a_{23}\neq 0$ and $a_{32}\neq 0$ and hence $1/\varepsilon\relstar \nu$. Since $1/\varepsilon\relstar\varepsilon$, this gives that $\varepsilon\relstar\nu$.\par
Now we assume the converse; that is, let $\varepsilon,\ \nu\in\F$ such that $\varepsilon\relstar\nu$. Let $A$ be the automorphism of 
$L_{5,8}$ represented by the diagonal matrix with the entries $(1, 1, \alpha, 1, \alpha)$ in the diagonal. Then $A\langle(0,1,0,1,0,\varepsilon)\rangle=\langle(0,1,0,1,0,\alpha^2\varepsilon)\rangle= \langle(0,1,0,1,0,\nu)\rangle$. Hence $L_{6,19}(\varepsilon)\cong 
L_{6,19}(\nu)$. \par

\subsection*{$\mathbf L_{5,9}$}

The group $\Aut(L)$ consists of the invertible matrices of the form   
\[A= \begin{pmatrix} a_{11} & a_{12} & 0 & 0 & 0 \\
a_{21} & a_{22} & 0 & 0 & 0\\
a_{31} & a_{32} & u & 0 & 0 \\
a_{41} & a_{42} & a_{11}a_{32}-a_{12}a_{31} & a_{11}u
& a_{12}u \\
a_{51} & a_{52} & a_{21}a_{32}-a_{22}a_{31} & a_{21}u 
& a_{22}u
\end{pmatrix},\]
where $u = a_{11}a_{22}-a_{12}a_{21}$. Then $Z^2(L,\F)=\langle\Delta_{1,2},\Delta_{1,3},\Delta_{1,4},\Delta_{1,5}+\Delta_{2,4},\Delta_{2,3},\Delta_{2,5}\rangle$ and $B^2(L,\F)=\langle\Delta_{1,2},\Delta_{1,3},\Delta_{2,3}\rangle$, and hence $H^2(L,\F)=\langle \overline{\Delta_{1,4}},\overline{\Delta_{1,5}}+\overline{\Delta_{2,4}},\overline{\Delta_{2,5}}\rangle$. Further,
\[\S=\{\langle (a,b,c)\rangle\mid ac-b^2\neq 0\}.\]
Let $\vartheta=(a,b,c)\in H^2(L,\F)$. Then $A\vartheta=(\bar a,\bar b,\bar c)$ where 
\begin{eqnarray*}
\bar a &=& (a_{11}^2a+2a_{11}a_{21}b+a_{21}^2c)u;\\
\bar b &=& (a_{11}a_{12}a+(a_{11}a_{22}+a_{12}a_{21})b+a_{21}a_{22}c)u;\\
\bar c &=& (a_{12}^2a+2a_{12}a_{22}b+a_{22}^2c)u.
\end{eqnarray*}
Choose $S=\langle(a,b,c)\rangle\in\S$. Let us consider three cases: $a\neq 0$; $a=0$ and $c\neq 0$; $a=0$, $c=0$, and $\char \F\neq 2$. Let, in these cases, 
$B$ denote the first, the second, or the third of the following transformations, respectively:
\[\begin{pmatrix}
 1&-b&0&0&0\\
0&a&0&0&0\\
0&0&a&0&0\\
0&0&0&a&-ab\\
0&0&0&0&a^2
\end{pmatrix},\quad
\begin{pmatrix}
  0&-c&0&0&0\\
1&b&0&0&0\\
0&0&c&0&0\\
0&0&0&0&-c^2\\
0&0&0&c&bc
\end{pmatrix},\quad
\begin{pmatrix}
 1&-1&0&0&0\\
1&1&0&0&0\\
0&0&2&0&0\\
0&0&0&2&-2\\
0&0&0&2&2
\end{pmatrix}.\]
We obtain, in the first and the second case, that $BS=\langle(1,0,ac-b^2)\rangle$, while, in the third case, we find that $BS=\langle(1,0,-1)\rangle$.\par
To summarize, if $\char \F\neq 2$ then, as $ac-b^2\neq 0$, the set formed by the subspaces $\langle (1,0,\varepsilon)\rangle$ with 
$\varepsilon\neq 0$  contain a representative in each of the $\Aut(L)$-orbits on $\mathcal S$. The Lie algebra corresponding to 
$\langle (1,0,\varepsilon)\rangle$ is $L_{6,21}(\varepsilon)$. If $\char \F=2$ then the set consisting of the subspaces $\langle(1,0,\varepsilon)\rangle$ and $\langle(0,1,0)\rangle$ contain such a system of representatives. The Lie algebra corresponding to  $\langle(0,1,0)\rangle$ is $L_{6,6}^{(2)}$. \par
The expressions for $\bar a,\ \bar b,\ \bar c$ give in characteristic~2 that $\langle (0,1,0)\rangle$ is fixed by $\Aut(L)$, and hence $L_{6,6}^{(2)}\not\cong L_{6,21}(\varepsilon)$. We claim, for $\varepsilon,\ \nu\in\F^*$ that $L_{6,21}(\varepsilon)\cong L_{6,21}(\nu)$ if and only if $\varepsilon\relstar\nu$. Suppose first that $L_{6,21}(\varepsilon)\cong L_{6,21}(\nu)$. Then there is some automorphism $A\in\Aut(L)$ such that $A(1,0,\varepsilon)=(1,0,\nu)$. Using the equations for $\bar a$, $\bar b$, and $\bar c$ we obtain that
\begin{eqnarray}
a_{11}a_{12}+a_{21}a_{22}\varepsilon&=&0;\label{eq591}\\
a_{12}^2+a_{22}^2\varepsilon-a_{11}^2\nu-a_{21}^2\varepsilon\nu&=&0\label{eq592}.
\end{eqnarray}
Simple computation shows that
\[(a_{12}a_{22}-\nu a_{11}a_{21})(a_{11}a_{12}+a_{21}a_{22}\varepsilon)-
a_{11}a_{22}(a_{12}^2+a_{22}^2\varepsilon-a_{11}^2\nu-a_{21}^2\varepsilon\nu)=
(a_{11}a_{22}-a_{12}a_{21})(a_{11}^2\nu-a_{22}^2\varepsilon).\]
Since $a_{11}a_{22}-a_{12}a_{21}\neq 0$, we obtain that equations~\eqref{eq591} and~\eqref{eq592} imply that $a_{11}^2\nu-a_{22}^2\varepsilon$. Then either $\varepsilon\relstar\nu$ or $a_{11}=a_{22}=0$. If $a_{11}=a_{22}=0$, then equation~\eqref{eq592} becomes $a_{12}^2=a_{21}^2\varepsilon\nu$. Since $A$ is invertible, we obtain, in this case, that $1/\varepsilon\relstar\nu$. Since $\varepsilon\relstar 1/\varepsilon$, this gives that $\varepsilon\relstar\nu$.\par
Conversely, let us suppose that $\varepsilon,\ \nu\in\F^*$ such that $\varepsilon\relstar\mu$. That is, there is some $\alpha\in\F^*$ such that $\nu=\alpha^2\varepsilon$. Then let $A$ be the automorphism of $L$ represented by the diagonal matrix with $(1,\alpha,\alpha,\alpha,\alpha^2)$ as its diagonal. Then $A\langle (1,0,\varepsilon)\rangle=\langle (1,0,\alpha^2\nu)\rangle= \langle (1,0,\nu)\rangle$.

\subsection*{$\mathbf L_{4,1}$}\label{41sec}

The automorphism group of $L$ is $\GL_4(\F)$ and $H^2(L,\F)=Z^2(L,\F)$ consists of all skew-symmetric bilinear forms on $L$, and hence $H^2(L,\F)=\langle \Delta_{1,2},\Delta_{1,3},\Delta_{1,4},\Delta_{2,3},\Delta_{2,4},\Delta_{3,4}\rangle$. Note that $H^2(L,\F)$ is naturally isomorphic to the wedge product  $\Lambda^2( \F^4)$ as $\GL_4(\F)$-modules; therefore we do not write the explicit formulas for the action. The set of allowable subspaces $\mathcal S$ consists of the 2-dimensional subspaces $S=\langle \vartheta_1,\vartheta_2\rangle$ such that $\vartheta_1^\perp\cap\vartheta_2^\perp=0$. Let us compute the  representatives of the $\Aut(L)$-orbits on $H^2(L,\F)$. Let $S\in\mathcal S$ and write 
\[S=\langle(a_1,b_1,c_1,d_1,e_1,f_1),(a_2,b_2,c_2,d_2,e_2,f_2)\rangle.\]
By Lemma~\ref{altlemma}(i), we may assume without loss of generality that $a_1=1$. Then
\[\begin{pmatrix}
1&0&d_1&e_1\\
0&1&-b_1&-c_1\\
0&0&1&0\\
0&0&0&1&
\end{pmatrix}S=\langle(1,0,0,0,0,f_1'),(0,b_2',c_2',d_2',e_2',f_2')\rangle,\]
where $f_1',\ b_2',\ c_2',\ d_2',\ e_2',\ f_2'\in\F$. Thus no generality is lost by assuming that
\[S=\langle(1,0,0,0,0,f_1),(0,b_2,c_2,d_2,e_2,f_2)\rangle.\]
We claim that there is $B\in \Aut(L)$ such that 
\[BS=\langle(1,0,0,0,0,f_1),(0,1,0,0,\bar e_2,f_2)\rangle.\]
Consider the following list of matrices:
\[\begin{pmatrix}
1&-d_2&0&0\\
0&b_2&0&0\\
0&0&1&-c_2\\
0&0&0&b_2&
\end{pmatrix}, 
\begin{pmatrix}
1&-e_2&0&0\\
0&c_2&0&0\\
0&0&0&-c_2\\
0&0&1&b_2&
\end{pmatrix},
\begin{pmatrix}
0&-d_2&0&0\\
1&b_2&0&0\\
0&0&1&-e_2\\
0&0&0&d_2&
\end{pmatrix},
\begin{pmatrix}
0&-e_2&0&0\\
1&c_2&0&0\\
0&0&0&-e_2\\
0&0&1&d_2&
\end{pmatrix}.\]
If $(b_2,c_2,d_2,e_2)\neq (0,0,0,0)$, then the list above contains at least one invertible matrix, and let $B$ denote this matrix. On the other hand, if $(b_2,c_2,d_2,e_2)=(0,0,0,0)$, then $f_2\neq 0$ and set, in this case,
\[B=\begin{pmatrix}
1&-f_2&f_1&0\\
1&0&0&0\\
0&0&1&0\\
-1&0&0&f_2&
\end{pmatrix}.\]
Having defined $B$ as above, we obtain $BS= \langle(1,0,0,0,0,f_1),(0,1,0,0,\bar e_2,f_2)\rangle$ with some $\bar e_2$, as claimed. 
Let us hence suppose without loss of generality that $S=\langle(1,0,0,0,0,f_1),(0,1,0,0,e_2,f_2)\rangle$. Next consider the following three cases: $f_1\neq 0$; $f_1=0$ and $e_2\neq 0$; $f_1=0$ and $e_2=0$. Note that in the last case $f_2\neq 0$, as otherwise $\vartheta_1^\perp\cap\vartheta_2^\perp\neq 0$. Define $C\in\Aut(L)$ in these cases, respectively, as follows:
\[\begin{pmatrix}
f_1&0&0&0\\
0&f_1&0&0\\
0&0&1&0\\
0&0&0&f_1
\end{pmatrix},\quad
\begin{pmatrix}
1&0&0&0\\
0&f_2&-1&0\\
0&-e_2&0&0\\
0&0&0&1
\end{pmatrix},\quad
\begin{pmatrix}
1&0&0&0\\
0&f_2&-1&0\\
0&0&1&0\\
0&0&0&1
\end{pmatrix}.\]
We obtain, in each of these cases, that $CS=\langle(1,0,0,0,0,1),(0,1,0,0,e_2',f_2)\rangle$ where $e_2'\in\F$. If $f_2=0$ then 
$CS=\langle(1,0,0,0,0,1),(0,1,0,0,\varepsilon,0)\rangle$ with some $\varepsilon\in\F$. On the other hand, if $f_2\neq 0$ then 
\[\begin{pmatrix}
f_2&0&0&0\\
0&1&0&0\\
0&0&f_2&0\\
0&0&0&1
\end{pmatrix}
\langle(1,0,0,0,0,1),(0,1,0,0,e_2',f_2)\rangle=\langle(1,0,0,0,0,1),(0,1,0,0,e_2'',1)\rangle,\]
where $e_2''\in\F$. Suppose now that $\char \F\neq 2$, and let $D$  be the first of the following two automorphisms if  $e_2''\in\{0,-1/4\}$,  while we let $D$ be the second otherwise:
\[\begin{pmatrix}
2e_2''&0&0&0\\
0&2&0&0\\
0&1&1&0\\
-1&0&0&4e_2''+1&
\end{pmatrix},\quad
\begin{pmatrix}
1&0&0&1\\
0&2&0&0\\
0& 1&1&0\\
0&0&0&2
\end{pmatrix}.\]
Then $D\langle(1,0,0,0,0,1),(0,1,0,0,e_2'',1)\rangle=\langle(1,0,0,0,0,1),(0,1,0,0,\varepsilon,0)\rangle$.\par
To summarize, in characteristic different from two, the set of 2-spaces $\langle(1,0,0,0,0,1),(0,1,0,0,\varepsilon,0)\rangle$ where
$\varepsilon\in\F$ contains at least one representative of each of the $\Aut(L)$-orbits on $\mathcal S$. In characteristic~2, such a set of 2-spaces is formed by the spaces $\langle(1,0,0,0,0,1),(0,1,0,0,\varepsilon,0)\rangle$ and the spaces $\langle(1,0,0,0,0,1),(0,1,0,0,\nu,1)\rangle$ where $\varepsilon,\ \nu\in\F$. The corresponding Lie algebras are $L_{6,22}(\varepsilon)$ and $L_{6,7}^{(2)}(\varepsilon)$, respectively.\par 
It remains to find the possible isomorphisms of the step-2 descendants of $L$. The group $\Aut(L)$ preserves, modulo scalars, a quadratic form $Q$ on $H^2(L,\F)$, defined, for  $\vartheta=(a,b,c,d,e,f)\in H^2(L,\F)$, as 
\begin{equation}\label{qfeq}
Q(\vartheta)=af-be+cd.
\end{equation}
Let $f$ denote the symmetric bilinear form associated with $Q$. It is easy to see that $Q$ is indeed preserved by the action of $\GL(V)$ modulo scalars; namely, for $A\in\GL(V)$ and $v\in V$ we have that $Q(Av)=(\det A)Q(v)$. \par
Assume first that $\char \F\neq 2$. and consider two subspaces $S_1=\langle(1,0,0,0,0,1),(0,1,0,0,\varepsilon,0)\rangle$ and $S_2=\langle(1,0,0,0,0,1),(0,1,0,0,\nu,0)\rangle$ such that $S_1$ and $S_2$ are in the same $\Aut(L)$-orbit. Since the determinants of the Gram matrices of the form $f$ restricted to $S_1$ and $S_2$ are $4\varepsilon$ and $4\nu$, respectively, Lemma~\ref{gramlemma}(2) implies that $\nu=\alpha^2\varepsilon$ with some $\alpha\in\F^*$. Conversely if $\nu=\gamma^2\varepsilon$ with some $\gamma\in\F^*$ and $A$ is the automorphism of $L_{4,1}$ represented by the diagonal matrix with the entries $(1,\gamma,1,\gamma)$ in the diagonal, then $AS_1=S_2$. \par
Suppose now that the characteristic of $\F$ is $2$. Set $S_1=\langle(1,0,0,0,0,1),(0,1,0,0,\varepsilon,0)\rangle$ and $S_2=\langle(1,0,0,0,0,1),(0,1,0,0,\nu,1)\rangle$ where $\varepsilon,\ \nu\in\F$. Since the restriction of $f$ is identically zero on $S_1$ while it is non-singular on $S_2$, we obtain that $S_1$ and $S_2$ cannot be in the same $\Aut(L)$-orbit. Suppose now that
 $S_1=\langle(1,0,0,0,0,1),(0,1,0,0,\varepsilon_1,0)\rangle$ and $S_2=\langle(1,0,0,0,0,1),(0,1,0,0,\varepsilon_2,0)\rangle$ 
such that $S_1$ and $S_2$ are in the same $\Aut(L)$-orbit. Since $f$ is identically zero on $S_1$ and $S_2$ and, for $i=1,\ 2$,  
$Q(1,0,0,0,0,1)Q(0,1,0,0,\varepsilon_i,0)=\varepsilon_i$, we obtain from Lemma~\ref{gramlemma}(iii) that $\varepsilon_2= \alpha^2\varepsilon_1+\beta^2$ with some $\alpha\in\F^*$ and $\beta\in\F$. Assume, conversely that $\varepsilon_2=\alpha^2\varepsilon_1+\beta^2$ with some $\alpha\in\F^*$ and $\beta\in\F$. Since there is nothing to prove if $\varepsilon_1=\varepsilon_2=1$, we may assume without loss of generality that $\varepsilon_1\neq 1$. Let $A$ be the automorphism of $L_{4,1}$ represented by the matrix 
\[\begin{pmatrix}
\varepsilon_1&\beta&1&\varepsilon_1\beta\\
0&\alpha&0&\alpha\\
1&\beta&1&\beta\\
0&\alpha&0&\varepsilon_1\alpha
\end{pmatrix}.\]
Then $\det A=\alpha^2(\varepsilon_1^2+1)$ which, by assumption, is non-zero, and so $A$ does define an isomorphism. Further, $A\langle(1,0,0,0,0,1),(0,1,0,0,\varepsilon_1,0)\rangle=\langle(1,0,0,0,0,1),(0,1,0,0,\varepsilon_2,0)\rangle$. \par
Finally if  $S_1=\langle(1,0,0,0,0,1),(0,1,0,0,\nu_1,1)\rangle$ and $S_2=\langle(1,0,0,0,0,1),(0,1,0,0,\nu_2,1)\rangle$, then the restriction of $f$ on $S_1$ and on $S_2$ is non-singular, and the given bases of $S_1$ and $S_2$ are symplectic. Further, the Arf invariants of $S_1$ and $S_2$ with respect to these bases are $\nu_1$ and $\nu_2$, respectively. Thus Lemma~\ref{gramlemma}(iii)
gives that $\nu_1+\nu_2\in\{\alpha^2+\alpha\ |\ \alpha\in\F^\ast\}$. Assume, conversely, that $\alpha^2+\alpha+\nu_1+\nu_2=0$ 
with some $\alpha$. Let $A$ denote the isomorphism of $L_{4,1}$ represented by the matrix
\[\begin{pmatrix}
1&0&0&\alpha\\
0&1&0&0\\
0&\alpha&1&0\\
0&0&0&1
\end{pmatrix}\]
Then $A\langle(1,0,0,0,0,1),(0,1,0,0,\nu_1,1)\rangle=\langle(1,0,0,0,0,1),(0,1,0,0,\nu_2,1)\rangle$. \par 
The argument presented in this section give rise to the proof of Theorem~\ref{kleinth}.

\begin{proof}[Proof of Theorem~$\ref{kleinth}$.]
Let $V$ be the vector space as in the statement of the theorem and let $L$ be the 4-dimensional abelian Lie algebra $L_{4,1}$.
As noted before, there is an isomorphism between the $\GL(4,\F)$-modules $V\wedge V$ and $H^2(L,\F)$ realized by the mapping 
$b_i\wedge b_j\mapsto\Delta_{i,j}$, and so we will identify $V\wedge V$ with $H^2(L,\F)$. If $S$ is a 2-dimensional subspace of $H^2(L,\F)$, then the corresponding central extension $L_S$ of $L$ is a 6-dimensional nilpotent Lie algebra with 4~generators and central derived subalgebra of dimension~2. In addition if $S$ is not allowable then $L_S=K\oplus\F$, which, using the classification of 5-dimensional nilpotent Lie algebras,  gives that $L_S\cong L_{5,8}\oplus\F$. Hence $\GL(4,\F)$ has a single orbit on the set of not allowable 2-dimensional subspaces. Since the orbits on the allowable 2-dimensional subspaces were determined in this section, the theorem follows. 
\end{proof}

\subsection*{$\mathbf L_{4,2}$}

The group $\Aut(L)$ consists of the invertible matrices of the form   
\[A= \begin{pmatrix} 
a_{11} & a_{12} & 0 & 0 \\
a_{21} & a_{22} & 0 & 0 \\
a_{31} & a_{32} & u & a_{34}\\
a_{41} & a_{42} & 0 & a_{44}
\end{pmatrix},\]
where $u = a_{11}a_{22}-a_{12}a_{21}$. We have $Z^2(L,\F)=\langle\Delta_{1,2},\Delta_{1,3},\Delta_{1,4},\Delta_{2,3},\Delta_{2,4}\rangle$, $B^2(L,\F)=\langle\Delta_{1,2}\rangle$, and so $H^2(L,\F)=\langle \overline{\Delta_{1,3}},\overline{\Delta_{1,4}},\overline{\Delta_{2,3}},\overline{\Delta_{2,4}}\rangle$. The set of allowable subspaces $\mathcal S$ consists of the 2-dimensional subspaces $S=\langle \vartheta_1,\vartheta_2\rangle$ such that $\vartheta_1^\perp\cap\vartheta_2^\perp\cap\langle x_3,x_4\rangle=0$. If $\vartheta=(a,b,c,d)\in H^2(L,\F)$, then $A\vartheta =(\bar a,\bar b,\bar c,\bar d)$ where
\begin{eqnarray*}
\bar a &=& (a_{11}a+a_{21}c)u;\\
\bar b &=& a_{11}a_{34}a + a_{11}a_{44}b + a_{21}a_{34}c+a_{21}a_{44}d;\\
\bar c &=& (a_{12}a+a_{22}c)u;\\
\bar d &=& a_{12}a_{34}a+a_{12}a_{44}b+a_{22}a_{34}c+a_{22}a_{44}d.
\end{eqnarray*}
Choose a 2-dimensional subspace $S=\langle \vartheta_1,\vartheta_2\rangle$ of $\mathcal S$ where $\vartheta_1=(a_1,b_1,c_1,d_1)$ 
and $\vartheta_2=(a_2,b_2,c_2,d_2)$. If $a_1=c_1=0$ and $a_2=c_2=0$ then $x_3\in\vartheta_1^\perp\cap\vartheta_2^\perp$, and hence $S$ is not allowable. Thus, by possibly swapping $\vartheta_1$ and $\vartheta_2$, we may assume without loss of generality that $(a_1,c_1)\neq(0,0)$. Let $B$ be the first of the following automorphisms if $a_1\neq 0$ and let $B$ be the second if $a_1=0$ (which implies that $c_1\neq 0$):
\[\begin{pmatrix}
1&-c_1&0&0\\
0&a_1&0&0\\
0&0&a_1&-a_1b_1\\
0&0&0&a_1^2&
\end{pmatrix},\quad
\begin{pmatrix}
0&1&0&0\\
-1&0&0&0\\
0&0&1&-d_1\\
0&0&0&c_1&
\end{pmatrix}.\]
Then the image $B  S$ is of the form $\langle (1,0,0,d_1'),(0,b_2',c_2',d_2')\rangle$ which implies that we may assume without loss of generality that $S=\langle (1,0,0,d_1),(0,b_2,c_2,d_2)\rangle$. We note that such an $S$ is allowable if and only if $(d_1,b_2,d_2)\neq(0,0,0)$.\par
Suppose first that $c_2=0$ and $d_2\neq 0$. Then 
\[\begin{pmatrix}
d_2&0&0&0\\
-b_2&1&0&0\\
0&0&d_2&d_1b_2\\
0&0&0&d_2
\end{pmatrix}S=
\langle (1,0,0,0),(0,0,0,1)\rangle.\]
Next, we assume that $c_2=0$ and $d_2=0$. If $d_1=0$ then $S=\langle (1,0,0,0),(0,1,0,0)\rangle$, while if $d_1\neq 0$ then
\[\begin{pmatrix}
d_1&0&0&0\\
0&1&0&0\\
0&0&d_1&0\\
0&0&0&d_1
\end{pmatrix}S=\langle (1,0,0,1),(0,1,0,0)\rangle.\]
Now suppose that $c_2\neq 0$ and $d_2=0$ and let $C$ be the first or the second of the following matrices, depending on whether $b_2=0$ or not:
\[\begin{pmatrix}
0&1&0&0\\
b_2&0&0&0\\
0&0&-b_2&0\\
0&0&0&-b_2c_2
\end{pmatrix},\quad 
\begin{pmatrix}
d_1&0&0&0\\
0&1&0&0\\
0&0&d_1&0\\
0&0&0&d_1
\end{pmatrix}.\]
Then $CS=\langle (1,0,0,1),(0,\varepsilon,1,0)\rangle$. Finally, assume that  $c_2\neq 0$ and $d_2\neq 0$ and let $D$ be the first 
or the second of the following matrices depending on whether $b_2\neq 0$:
\[\begin{pmatrix}
d_2&-d_2&0&0\\
-b_2&0&0&0\\
0&0&-b_2d_2&0\\
0&0&0&-b_2c_2
\end{pmatrix},\quad
\begin{pmatrix}
c_2&(d_1-1)c_2&0&0\\
0&d_2&0&0\\
0&0&c_2d_2&0\\
0&0&0&c_2^2
\end{pmatrix}.\]
Then $DS=\langle (1,0,0,1),(0,\bar b_2,1,1)\rangle$, with $\bar b_2\in\F$. If $\char \F\neq 2$ then this gives no new orbit as
\[\begin{pmatrix}
2&0&0&0\\
1&1&0&0\\
0&0&2&-2\\
0&0&0&4
\end{pmatrix}\langle (1,0,0,1),(0,\bar b_2,1,1)\rangle=\langle 
(1,0,0,1),(0,4\bar b_2+1,1,0)\rangle,\]
If $\char \F=2$ then we obtain that $S=\langle(1,0,0,1),(0,\nu,1,1)\rangle$.\par
To summarize, if $\char \F\neq 2$ then the list of 2-spaces $\langle (1,0,0,0),(0,0,0,1)\rangle$, $\langle (1,0,0,0),(0,1,0,0)\rangle$, $\langle (1,0,0,1),(0,1,0,0)\rangle$, and $\langle (1,0,0,1),(0,\varepsilon,1,0)\rangle$ with $\varepsilon\in\F$ 
contains at least one representative for each of the $\Aut(L)$-orbits on $\mathcal S$. The corresponding Lie algebras are $L_{6,27}$, $L_{6,25}$, $L_{6,23}$, and $L_{6,24}(\varepsilon)$. In characteristic~2, such a set is formed by the subspaces above in 
addition to the subspaces $\langle (1,0,0,1),(0,\nu,1,1)\rangle$ with $\nu\in\F$. The Lie algebra that corresponds to the subspace 
$\langle (1,0,0,1),(0,\nu,1,1)\rangle$ is $L_{6,8}^{(2)}(\nu)$. \par
Finally, in this section we have to determine the possible isomorphisms of the Lie algebras in the previous paragraph. First we note, for $L=L_{6,24}(\varepsilon)$ and $L=L_{6,8}^{(2)}(\nu)$,  that $C(L)=L^3$, while this equation is not valid for $L_{6,27}$, $L_{6,23}$, or $L_{6,25}$. In order to separate the Lie algebras $L_{6,27}$, $L_{6,25}$, $L_{6,23}$ we use the geometry of the $\Aut(L)$-action on $H^2(L,\F)$. The expressions for $\bar a$, $\bar b$, $\bar c$, and $\bar d$ above give that the action of the automorphism $A$ on $H^2(L,\F)$ is represented, with respect to the basis  $\{\overline{\Delta_{1,3}},\overline{\Delta_{1,4}},\overline{\Delta_{2,3}}, \overline{\Delta_{2,4}}\}$, by the matrix 
\begin{equation}\label{tensor}
\begin{pmatrix}
a_{11}u & a_{11}a_{34} & a_{12}u & a_{12}a_{34}\\
0 &  a_{11}a_{44} & 0 & a_{12}a_{44} \\
a_{21}u & a_{21}a_{34} & a_{22}u & a_{22}a_{34} \\
0 & a_{21}a_{44} & 0 & a_{22}a_{44}
\end{pmatrix}=
\begin{pmatrix}
a_{11} & a_{12} \\ 
a_{21} & a_{22}
\end{pmatrix}\otimes 
\begin{pmatrix}
u & a_{34} \\
0 & a_{44}
\end{pmatrix}.
\end{equation}
It is well-known that $\GL(2,\F)\otimes\GL(2,\F)$ preserves a quadratic form modulo scalars in its natural action on $\F^4$. 
To exploit this fact in our situation, define a quadratic form $Q$ on $H^2(L,\F)$ as \[Q(\vartheta)=\alpha_1\alpha_4-\alpha_2\alpha_3\quad\mbox{where}\quad\vartheta=(\alpha_1,\alpha_2,\alpha_3,\alpha_4).\]
Let $A\in\Aut(L)$ and decompose $A$ as $g_1\otimes g_2$ as in~\eqref{tensor}. Then we have, for all $\vartheta\in H^2(L,\F)$, that
\[Q(A\vartheta)=Q((g_1\otimes g_2)\vartheta)=(\det g_1)(\det g_2)Q(\vartheta).\]

Now to show that the subspaces obtained above are in different $\Aut(L)$-orbits, notice that the subspace $\langle (1,0,0,0),(0,1,0,0)\rangle$  is totally singular. On the other hand, the singular vectors of the subspace $\langle (1,0,0,1),(0,1,0,0)\rangle$ are the elements of the 1-space $\langle (0,1,0,0)\rangle$, and hence two singular vectors are linearly dependent. Moreover, the singular vectors of the subspace $\langle (1,0,0,0),(0,0,0,1)\rangle$ are the elements of the 1-spaces $\langle (1,0,0,0)\rangle$ and $\langle (0,0,0,1)\rangle$, which shows that there is a pair of linearly independent singular vectors. Thus these three subspaces are in different $\Aut(L)$-orbits. Let us consider now two subspaces of the form 
$S_1=\langle (1,0,0,1),(0,\varepsilon,1,0)\rangle$ and $S_2=\langle (1,0,0,1),(0,\nu,1,0)\rangle$ and assume that they are in the same $\Aut(L)$-orbit; that is $S_2=AS_1$ with some $A\in\Aut(L)$. If $\char \F\neq 2$, then the Gram determinants of $S_1$ and $S_2$ with respect to the given bases are $\varepsilon$ and $\nu$, respectively. Lemma~\ref{gramlemma}(i) gives that $\varepsilon\relstar\nu$. Conversely, assume that $\char \F\neq 2$ and that $\varepsilon\relstar\nu$; that is $\nu=\varepsilon\alpha^2$ with some $\alpha\in\F^*$. Let $A$ be the diagonal automorphism of $L$ with $(\alpha,1,\alpha,\alpha^2)$ in the diagonal. Then $A\langle (1,0,0,1),(0,\varepsilon,1,0)\rangle=\langle (1,0,0,1),(0,\nu,1,0)\rangle$. Hence the subspaces $\langle (1,0,0,1),(0,\varepsilon,1,0)\rangle$ and $\langle (1,0,0,1),(0,\nu,1,0)\rangle$ are in the same $\Aut(L)$-orbit if
and only if $\varepsilon\relstar\nu$. This settles the isomorphisms among the possible step-2 descendants of $L_{4,2}$ in the cases when $\char \F\neq 2$. \par
Suppose next that $\char \F=2$ and let 
\[S_1=\langle (1,0,0,1),(0,\varepsilon_1,1,0)\rangle\quad\mbox{and}\quad S_2=\langle (1,0,0,1),(0,\varepsilon_2,1,0)\rangle,\]
as above. Since $f$ is identically zero on $S_1$ and $S_2$ and, for $i=1,\ 2$, $Q(1,0,0,0,0,1)Q(0,1,0,0,\varepsilon_i,0)=\varepsilon_i$, we obtain from Lemma~\ref{gramlemma}(iii) that $\varepsilon_2=\alpha^2\varepsilon_1+\beta^2$ with some $\alpha\in\F^*$ and $\beta\in\F$. Suppose conversely that $\char \F=2$ and that $\varepsilon_1\relstarplus\varepsilon_2$; that is $\varepsilon_2=\alpha^2\varepsilon_1+\beta^2$ with some $\alpha\in\F^*$ and $\beta\in\F$. Let $A$ be the automorphism of $L$ represented by the matrix
\[\begin{pmatrix}
\alpha & 0 & 0 & 0 \\
\beta & 1 & 0 & 0\\
0 & 0 & \alpha & \alpha\beta\\
0 & 0 & 0 & \alpha^2
\end{pmatrix}.\]
Then $AS_1=S_2$, which gives, when $\char \F=2$, that $\langle (1,0,0,1),(0,\varepsilon_1,1,0)\rangle$ and $\langle (1,0,0,1),(0,\varepsilon_2,1,0)\rangle$ are in the same $\Aut(L)$-orbit if and only if $\varepsilon_1\relstarplus\varepsilon_2$. \par
Suppose now that $S_1=\langle (1,0,0,1),(0,\nu_1,1,1)\rangle$ and $S_2=\langle (1,0,0,1),(0,\nu_2,1,1)\rangle$ and assume  that $S_1$ and $S_2$ are in the same $\Aut(L)$-orbit; that is, there is some $A\in\Aut(L)$ such that $S_1A=S_2$. The restriction of $f$ on $S_1$ and on $S_2$ is non-singular, and the given bases of $S_1$ and $S_2$ are symplectic. Further, the Arf invariants of $S_1$ and $S_2$ with respect to these bases are $\nu_1$ and $\nu_2$, respectively. Thus Lemma~\ref{gramlemma}(iii) gives that $\nu_1+\nu_2\in\{\alpha^2+\alpha\ |\ \alpha\in\F\}$. Conversely, suppose that $\nu_1,\ \nu_2\in\F$ such that $x^2+x+\nu_1+\nu_2$ has a solution $\alpha$ in $\F$. Let $A$ denote the automorphism represented by the matrix
\[\begin{pmatrix}
1 & 0 & 0 & 0 \\
\alpha & 1 & 0 & 0\\
0 & 0 & 1 & \alpha\\
0 & 0 & 0 & 1
\end{pmatrix}.\]
Then  $AS_1=S_2$, and so two subspaces $\langle(1,0,0,1),(0,\nu_1,1,1)\rangle$ and $\langle (1,0,0,1),(0,\nu_2,1,1)\rangle$ are in the same $\Aut(L)$-orbit if and only if the polynomial $x^2+x+\nu_1+\nu_2$ has a root in $\F$.

\section*{Acknowledgment}

The first and the third authors would like to acknowledge the 
support of the grants ISFL-1-143/BPD/2009 and PTDC/MAT/101993/2008, 
respectively, of the  {\em Funda\c c\~ao para a Ci\^ encia e a Tecnologia} (Portugal). 
The third author was also supported by
the Hungarian Scientific Research Fund (OTKA) grant~72845.

\bigskip

\noindent Serena Cical\`o and Csaba Schneider\\
Centro de \'Algebra de Universidade de Lisboa\\
Av.\ Professor Gama Pinto, 2, 1649-003 Lisboa, Portugal.\\
cicalo@science.unitn.it, csaba.schneider@gmail.com

\bigskip

\noindent Serena Cical\`o (current address)\\
Dipartimento di Matematica e Informatica\\
Via Ospedale, 72 - 09124 Cagliari, Italy\\
cicalo@science.unitn.it

\bigskip

\noindent Willem A.\ de Graaf\\
Dipartimento di Matematica\\
Via Sommarive 14 - 38050 Povo (Trento), Italy\\
degraaf@science.unitn.it


\begin{thebibliography}{10}

\bibitem{Magma}
W.~Bosma, J.~Cannon, and C.~Playoust.
\newblock The {M}agma algebra system {I}: {T}he user language.
\newblock \emph{J. Symbolic Comput.}, 24(3-4):235--265, 1997.

\bibitem{liealgdb} Serena Cical\`o, W. A. de~Graaf, and Csaba Schneider. {\sf LieAlgDB} -- a GAP package, Version 2.1, November 2010, 
www.sztaki.hu/$\sim$schneider/Research/LieAlgDB.

\bibitem{cohn}
P. M. Cohn.
{\em Basic Algebra: Groups, Rings and Fields}.
Springer, 2003.

\bibitem{prew}W.~A. de~Graaf.
\newblock Classification of 6-dimensional nilpotent {L}ie algebras over
              fields of characteristic not 2 (preprint),
\newblock  http://arxiv.org/abs/math.RA/0511668, 2005.

\bibitem{artw}
W.~A. de~Graaf.
\newblock Classification of 6-dimensional nilpotent {L}ie algebras over
              fields of characteristic not 2.
\newblock {\em J. Algebra}, 309:640--653, 2007.

\bibitem{eickob}
B.~Eick and E.~A. O'Brien.
\newblock Enumerating {$p$}-groups.
\newblock {\em J. Austral. Math. Soc. Ser. A}, 67(2):191--205, 1999.

\bibitem{gap}
The GAP Group. {\sf GAP} -- Groups, Algorithms, and Programming, 
Version 4.4.12; 2008. (http://www.gap-system.org)

\bibitem{gong}
M.-P. Gong.
\newblock {\em Classification of Nilpotent Lie Algebras of Dimension $7$}.
\newblock Ph.D. thesis, University of Waterloo, Waterloo, Canada, 1998.

\bibitem{lang}
Serge Lang.
\newblock {\em Algebra.}
\newblock 
Revised third edition. Graduate Texts in Mathematics, 211. Springer-Verlag, New York, 2002. 

\bibitem{morozov}
V.~V. Morozov.
\newblock Classification of nilpotent {L}ie algebras of sixth order.
\newblock {\em Izv. Vys\v s. U\v cebn. Zaved. Matematika}, 1958(4
  (5):161--171, 1958.

\bibitem{newobrvau}
M.~F. Newman, E.~A. O'Brien, and M.~R. Vaughan-Lee.
\newblock Groups and nilpotent {L}ie rings whose order is the sixth power of a
  prime.
\newblock {\em J. Algebra}, 278(1):383--401, 2004.

\bibitem{obrien}
E.~A. O'Brien.
\newblock The {$p$}-group generation algorithm.
\newblock {\em J. Symbolic Comput.}, 9(5-6):677--698, 1990.

\bibitem{obvau}
E.~A. O'Brien and M.~R. Vaughan-Lee.
\newblock The groups with order {$p\sp 7$} for odd prime {$p$}.
\newblock {\em J. Algebra}, 292(1):243--258, 2005.

\bibitem{sch}
C.~Schneider.
\newblock A computer-based approach to the classification of nilpotent {L}ie
  algebras.
\newblock {\em Experiment. Math.}, 14(2):153--160, 2005.

\bibitem{skjelsund}
Tor Skjelbred and Terje Sund.
\newblock Sur la classification des alg\`ebres de {L}ie nilpotentes.
\newblock {\em C. R. Acad. Sci. Paris S\'er. A-B}, 286(5), 1978.

\bibitem{stroppel}
Markus Stroppel.
\newblock
The Klein quadric and the classification of nilpotent {L}ie algebras
of class two. 
\newblock {\em J. Lie Theory}, 18:391--411, 2008.

\bibitem{umlauf}
K.~A. Umlauf.
{\em \"Uber die Zusammensetzung der endlichen continuierlichen Transformationsgruppen
insbesondere der Gruppen vom Range Null.}
Ph.D. thesis, University of Leipzig, Germany, 1891.


\end{thebibliography}
\end{document}